\documentclass[11pt]{article}

\usepackage{enumerate}
\usepackage{mathrsfs}
\usepackage{amsmath, color}
\usepackage{latexsym}
\usepackage{amssymb}
\usepackage{amsthm}
\usepackage{amsfonts}
\usepackage{dsfont}
\usepackage{hyperref}
\usepackage{xcolor,cancel}

\newtheorem{dfn}{Definition}[section]
\newtheorem{thm}[dfn]{Theorem}
\newtheorem{rmk}[dfn]{Remark}

\newtheorem{prop}[dfn]{Proposition}
\newtheorem{co}[dfn]{Corollary}
\newtheorem{lem}[dfn]{Lemma}
\theoremstyle{definition}

\numberwithin{equation}{section}

\newcommand{\IE}{{\mathbb{E}}}
\newcommand{\IP}{{\mathbb{P}}}
\newcommand{\IR}{{\mathbb{R}}}
\newcommand{\FF}{{\mathcal{F}}}
\newcommand{\EE}{{\mathcal{E}}}

\newcommand{\IZ}{{\mathbb{Z}}}

\def\eps{\varepsilon}
\def\wh{\widehat}
\def\wt{\widetilde}
\def\<{\langle}
\def\>{\rangle}

\title{\bf  Discrete Approximation to Brownian Motion with Varying Dimension in Unbounded Domains}
\date{\today}
\author{{\bf Shuwen Lou}}

\begin{document}

\maketitle

\begin{abstract}
We establish the discrete approximation to Brownian motion with varying dimension (BMVD in abbreviation)  by  random walks. The setting is very similar to that in \cite{Lou1}, but here we use a different method allowing us to get rid the restrictions in \cite{Lou1} (or \cite{BC}) that the underlying state space has to be bounded, and that the initial distribution of the limiting continuous process has to be its invariant distribution. The approach in this paper is that we first obtain heat kernel upper bounds for the approximating random walks that are uniform in their mesh size, by establishing a Nash-type inequality based on their Dirichlet form characterization. Using the heat kernel upper bound,  we then show the tightness of the approximating random walks by delicate analysis.

\end{abstract}



\section{Introduction}\label{Intro}

Brownian motion on spaces with varying dimension was introduced  in \cite{CL}. The state space of such a process  looks like a plane with a vertical half line installed on it,  ``embedded" in the following space:
\begin{equation*}
\IR^2 \cup \IR_+ =\{(x_1, x_2, x_3)\in \IR^3: x_1=0 \textrm{ or } x_2=x_3=0 \hbox{ and } x_1>0\}.
\end{equation*}
As has been noted in \cite{CL}, Brownian motion cannot be defined on such a state space in the usual sense because a two-dimensional Brownian motion does not hit a singleton. The BMVD in \cite{CL} was constructed by ``shorting" a closed disc on $\IR^2$ to a singleton, which in other words, makes the  resistance on this closed disc zero, so that  the process travels on the disc at infinite velocity. The resulting Brownian motion hits the shorted disc in finite time with probability one. Then an infinite half line $\IR_+$  is attached to the plane $\IR^2$  at this ``shorted" disc.

The  state space of BMVD can  be rigorously defined as follows:   Fix $0<\eps<1/64$ and denote by  $B_\eps $  the closed disk on $\IR^2$ centered at $(0,0)$ with radius $\eps $. Let 
${D_\eps}:=\IR^2\setminus  B_\eps $. By identifying $B_\eps $ with a singleton denoted by $a^*$, we  introduce a topological space $E:={D_\eps}\cup \{a^*\}\cup \IR_+$, with the origin of $\IR_+$ identified with $a^*$ and
a neighborhood of $a^*$ defined as $\{a^*\}\cup \left(V_1\cap \IR_+ \right)\cup \left(V_2\cap {D_\eps}\right)$ for some neighborhood $V_1$ of $0$ in $\IR^1$ and $V_2$ of $B_\eps $ in $\IR^2$. Let $m$ be the measure on $E$ whose restriction on $\IR_+$ or ${D_\eps}$ is $1$- or $2$-dimensional Lebesgue measure, respectively.  In particular,   we set   $m (\{a^*\} )=0$. Note that the measure $m$  depends on $\eps$, the radius of the ``hole" $B_\eps$.

Same as in \cite{CL},  the state space $E$ is equipped with the geodesic distance  $\rho$.  Namely, for $ x,y\in E$, $\rho (x, y)$ is the shortest path distance (induced
from the Euclidean space) in $E$ between $x$ and $y$.
For notation simplicity, we write $|x|_\rho$ for $\rho (x, a^*)$.
We use $| \cdot |$ to denote the usual Euclidean norm. For example, for $x,y\in {D_\eps}$,
$|x-y| $ is  the Euclidean distance between $x$ and $y$ in $\IR^2$.
Note that for $x\in {D_\eps}$, $|x|_\rho =|x|-\eps$.
Clearly,
\begin{equation}\label{e:1.1}
\rho (x, y)=|x-y|\wedge \left( |x|_\rho + |y|_\rho \right)
\quad \hbox{for } x,y\in {D_\eps}
\end{equation}
and $\rho (x, y)= |x|+|y|-\eps$ when $x\in \IR_+$ and $y\in {D_\eps}$ or vice versa.
Here and in the rest of this paper, for $a,b\in \IR$, $a\wedge b:=\min\{a,b\}$.

The following definition for BMVD can be found in \cite[Definition 1.1]{CL}.
\begin{dfn}[Brownian motion with varying dimension]\label{def-bmvd}An $m$-symmetric diffusion process satisfying the following properties is called Brownian motion with varying dimension. 
\begin{description}
\item{\rm (i)} its part process in $\IR_+$ or $D_\eps$ has the same law as standard Brownian motion in $\IR_+$ or $D_\eps$;
\item{\rm (ii)} it admits no killings on $a^*$;
\end{description}
\end{dfn}
 It follows from the definition that BMVD spends zero amount of time under Lebesgue measure (i.e. zero sojourn time) at $a^*$. The following theorem gives the  Dirichlet form characterization of BMVD. 

\begin{thm}[\cite{CL}]\label{BMVD-non-drift}
For every  $\eps >0$,
BMVD  on $E$ with parameter $\eps$ exists and is unique.
Its associated Dirichlet form $(\EE, \mathcal{D}(\EE))$ on $L^2(E; m)$ is given by
\begin{eqnarray*}
\mathcal{D}(\EE) &= &  \left\{f: f|_{D_\eps}\in W^{1,2}(D_\eps),  \, f|_{\IR_+}\in W^{1,2}(\IR_+),
\hbox{ and }
f (x) =f (0) \hbox{ q.e. on } {\partial D_\eps}\right\},  
\\
\EE(f,g) &=& \frac{1}{4} \int_{D_\eps}\nabla f(x) \cdot \nabla g(x) dx+\frac{1}{2}\int_{\IR_+}f'(x)g'(x)dx . 
\end{eqnarray*}
Furthermore, such a BMVD is a Feller process with strong Feller property. 
\end{thm}

\begin{rmk}
For the computation convenience in this paper, we let  BMVD in Theorem \ref{BMVD-non-drift} corresponds to the BMVD defined in \cite[Theorem 2.2]{CL} with parameters $(\eps, p=1)$ but running on $D_\eps$ at a speed $1/2$. 
\end{rmk}

It is well-known that Brownian motion on Euclidean spaces is the scaling limit of simple random walks on square lattices. In \cite{BC}, it was shown that reflected Brownian motion in bounded domains can be approximated by simple random walks. Roughly speaking, the method in \cite{BC} consists of two steps: The first step is to prove the tightness of the laws of the random walks by verifying the tightness of the martingale parts of the random walks through the analysis of their quadratic variations and then applying the Lyons-Zheng decomposition. The second step is to show the uniqueness of the subsequential limits of the laws of the random walks by characterizing the limits as the solution to the martingale problem for the infinitesimal generator of the continuous process. Utilizing the same method, it was proved in \cite{Lou1} that BMVD killed upon exiting a bounded domain can be approximated weakly by simple random walks in lattices with varying dimension. However, the method used in both \cite{BC} and \cite{Lou1} has  the  limitation that it only works on bounded domains, and that the initial distribution has to be the invariant measure  of the  continuous limiting process.

In this paper, we use a different approach to get rid of the constrain that the approximation can only be established on bounded domains with initial distribution having to be the invariant measure. To the best knowledge of the author, this method has not been used in any existing  literature, but can potentially be adapted to discrete approximations to other Markov processes with nice Dirichlet form characterizations.

Same as in \cite{Lou1}, we use a sequence of approximating random walks indexed by $k\ge 1$ on lattices with  with mesh-size $2^{-k}$.   Notice that  in \cite{Lou1}, these approximating random walks can be characterized in terms of Dirichlet forms.  Therefore carefully applying  the combination of  Nash-type inequality and Davies method provides some heat kernel upper bound that is ``uniform in $k$" for the entire  family of  random walks. Intuitively speaking,  this gives some level of ``equi-continuity" for the transition densities of this family random walks. From there.  the C-tightness of the random walks can be established via some delicate analysis. As a result, we show that starting from the darning point $a^*$, BMVD on $E$  can be weakly approximated by a family of random walks with varying dimension starting from the respective darning point in each of their state spaces. We note that although in this paper, the approximation   is only established for BMVD starting from the darning point $a^*$,  with similar computation one can show the same approximation results for BMVD starting from any single point.

The rigorous description of the  state spaces of   random walks with varying dimension has been given in \cite{Lou1}. Here we repeat it for completion.  For $k\in \mathbb{N}$,  let $D_\eps ^k:=D_\eps \cap 2^{-k}\IZ^2$.  We identify vertices  of $2^{-k}\IZ^2$ that are contained in the closed disc $B_\eps$ as a singleton $a^*_k$. Let $E^k:=2^{-k}\IZ_+\cup \{a^*_k\}\cup D_\eps^k$, where $\IZ_+=\{1,2,\dots\}$.

Recall that in general, a graph $G$  can be written as ``$G=\{G_v, G_e\}$", where $G_v$ is its collection of vertices, and $G_e$ is its connection of edges. Given any two vertices in $a,b\in G$,  if there is unoriented edge with endpoints $a$ and $b$,  we say $a$ and $b$ are adjacent to each other in $G$, written  `` $a\leftrightarrow b$ in $G$".   One can  always assume   that given two vertices $a, b$ on a graph, there is at most  one such unoriented edge connecting these two points (otherwise edges with same endpoints can be removed and replaced with one single edge). This unoriented edge is denoted by $e_{ab}$ or $e_{ba}$ ($e_{ab}$ and $e_{ba}$ are viewed as the same elelment in $G_e$).  In this paper,  for notational convenience, we denote by $\mathcal{G}_2:=\{2^{-k}\IZ^2, \mathcal{V}_2\}$, where $\mathcal{V}_2$ is the collection of the edges of $2^{-k}\IZ^2$. Also we denote by $\mathcal{G}_1:=\{2^{-k}\IZ_+\cup \{0\}, \mathcal{V}_1\}$ the  $1$-dimensional  lattice over $2^{-k}\IZ_+\cup \{0\}$, where $\mathcal{V}_1$ is the  collection of edges of $2^{-k}\IZ^1_+\cup \{0\}$. 

Now we introduce the graph structure on $E^k$. Let $G^k=\{G^k_v, G^k_e\}$ be a graph where $G^k_v=E^k$ is the collection of vertices and $G^k_e$ is the collection of unoriented edges over $E^k$  defined as follows:
\begin{align*}
G^k_e:=&\{e_{xy}:\, \exists \,x,y\in D_\eps^k, |\,x-y|=2^{-k},\, e_{xy}\in\mathcal{V}_2,\, e_{xy}\cap B_\eps=\emptyset\}
\\
\cup &\{e_{xy}: \exists \, x,y\in 2^{-k}\IZ_+\cup \{0\}, \,|x-y|=2^{-k},\, e_{xy}\in 
\mathcal{V}_1\}
\\
\cup &\{e_{xa^*_k}:x\in  D^k_\eps,\text{ there is at least one }y\in  2^{-k}\IZ^2\cap B_\eps \text{ with }|x-y|=2^{-k}, \,e_{xy}\in \mathcal{V}_2\}.
\end{align*}
  Note that $G^k=\{G^k_v, G^k_e\}$ is a connected graph. We emphasize that  given any $x\in G^k_v$, $x\neq a^*_k$, there is at most one element in $G^k_e$ with endpoints $x$ and $a^*_k$.   Denote by $v_k(x)=\#\{e_{xy}\in G^k_e\}$, i.e., the number of vertices in $G^k_v$ adjacent to $x$.  $E^k$ is equipped with the following underlying reference measure:
\begin{equation}\label{def-mk}
m_k(x):=\left\{
    \begin{aligned}
         &\frac{2^{-2k}}{4}v_k(x),  &x \in D^k_\eps;\\
        &\frac{2^{-k}}{2}v_k(x), &x \in 2^{-k}\IZ_+;\\
        & \frac{2^{-k}}{2}+\frac{2^{-2k}}{4}\left(v_k(x)-1\right), & x=a^*_k.
    \end{aligned}
\right.
\end{equation}
Next we define the random walks that will be shown to approximate the BMVD. Consider the following  Dirichlet form on $L^2(E^k, m_k)$:
\begin{align}\label{DF-RWVD-form}
\left\{
\begin{aligned}
&\mathcal{D}(\EE^{k})=L^2(E^k, m_k)
\\
&\EE^{k}(f, f)= \frac{1}{8}\sum_{\substack{e^o_{xy}:\; e_{xy}\in G^k_e,\\ x,y\in D^k_\eps\cup \{a_k^*\} }} \left(f(x)-f(y)\right)^2 +\frac{2^k}{4}\sum_{\substack{e^o_{xy}:\;e_{xy}\in G^k_e,\\ x,y\in 2^{-k}\IZ_+\cup \{a_k^*\} }}\left(f(x)-f(y)\right)^2,
\end{aligned}
\right.
\end{align}
where $e^o_{xy}$ is an {\it oriented edge} from  $x $  to  $y$. In other words, given any pair of  adjacent  vertices $x, y \in G^k_v$,  the edge with endpoints $x$ and $y$ is represented twice in the sum: $e^o_{xy}$ and $e^o_{yx}$.  One can verify that $(\EE^k, \mathcal{D}(\EE^{k}))$ on $L^2(E^k, m_k)$ is a regular symmetric Dirichlet form, therefore  we denote the  symmetric strong Markov process associated with it   by $X^k$.  The explicit distribution of $X^k$ is presented in Proposition \ref{jump-distribution-Xk}.   In this paper, for every fixed $0<\eps<1/64$, we select and then fix  some $k_0\in \mathbb{N}$ only depending on $\eps$ such that 
\begin{equation}\label{def-k0}
 2^{-k}<\eps/4 \quad \text{for all }    k\ge k_0.
\end{equation}

  Our main result is the following theorem.  
\begin{thm}\label{main-result}
For every $T>0$, the laws of  $\{X^{k}, \IP^{m_k}\}_{k\ge k_o}$ are tight in the space $\mathbf{D}([0, T], E, \rho)$ equipped with Skorokhod topology. Furthermore, as $k\rightarrow \infty$,  $(X^{k}, \IP^{\overline{m}_k}) $ converges weakly to BMVD with parameter $\eps$. 
\end{thm}

The rest of this paper is   organized as follows: For Section  \ref{S:2}, we first give a brief introduction to continuous-time reversible pure jump proceses and their corresponding symmetric Dirichlet forms in \S\ref{S:2.1}. Then in \S\ref{S:2.2} we present some basics about  the approximating random walks $\{X^k,k\ge 1\}$, including their explicit transition probabilities. These results werer obtained in \cite{Lou1}.  In \S\ref{S:2.3}, we review the results on isoperimetric inequalities for weighted graphs summarized from \cite{MB}.   In Section \ref{S:3}, we first establish Nash-type inequality for $\{X^k,k\ge 1\}$. From there using Davies method we obtain some heat kernel upper bounds for this family of random walks.  The tightness of $\{X^k,k\ge 1\}$ is shown in  Section \ref{S:4}, which is done with some very delicate analysis based upon the heat kernel upper bounds. Finally, the weak approximation result is completed in Section \ref{S:5}, by identifying
the limit of $\{X^k,k\ge 1\}$   as the unique solution to the martingale problem for the infinitesimal generator of BMVD, i.e., the generator associated with the Dirichlet form $(\EE, \mathcal{D}(\EE))$ in Theorem \ref{BMVD-non-drift}.

In this paper we  follow the convention that in the statements of the theorems or propositions $C, C_1, \cdots$ denote positive constants, whereas in their proofs $c, c_1, \cdots$ denote positive constants whose exact value is unimportant and may change
 from line to line.  

\section{Preliminaries}\label{S:2}

\subsection{Continuous-time reversible pure jump processes and symmetric Dirichlet forms}\label{S:2.1}

In this section, we give a brief background on continuous-time reversible pure jump processes   and symmetric Dirichlet forms. The results in this section can be found in \cite[\S 2.2.1]{CF}.

Suppose $\mathsf{E}$ is a locally compact separable metric space and $\{Q(x, dy)\}$ is a probability kernel on $(\mathsf{E}, \mathcal{B}(\mathsf{E}))$ with $Q(x, \{x\})=0$ for every $x\in \mathsf{E}$. Given a constant $\lambda >0$, we can construct a pure jump Markov process $\mathsf{X}$ as follows: Starting from $x_0\in \mathsf{E}$, $\mathsf{X}$ remains at $x_0$ for an exponentially distributed holding time $T_1$ with parameter $\lambda(x_0)$ (i.e., $\IE[T_1]=1/\lambda(x_0)$), then it jumps to some $x_1\in \mathsf{E}$ according to distribution $Q(x_0, dy)$; it remains at $x_1$ for another exponentially distributed holding time $T_2$  also with  parameter $\lambda(x_1)$ before jumping to $x_2$ according to distribution $Q(x_1, dy)$.  $T_2$ is independent of $T_1$. $\mathsf{X}$ then continues. The probability kernel $Q(x, dy)$ is called the {\it road map} of $\mathsf{X}$, and the $\lambda(x)$ is its {\it speed function}. If there is a $\sigma$-finite measure $\mathsf{m}_0$ on $\mathsf{E}$ with supp$[\mathsf{m}_0]=\mathsf{E}$ such that
\begin{equation}\label{symmetrizing-meas}
Q(x, dy)\mathsf{m}_0(dx)=Q(y, dx)\mathsf{m}_0(dy),
\end{equation}
$\mathsf{m}_0$ is called a {\it symmetrizing measure} of the road map $Q$. Another way to view \eqref{symmetrizing-meas} is that, $Q(x, dy)$ is the one-step transition ``probability" distribution, so its density with respect to the symmetrizing measure $Q(x,dy)/\mathsf{m}_0(dy)$ must be symmetric in $x$ and $y$, i.e., 
\begin{equation*}
\frac{Q(x, dy)}{\mathsf{m}_0(dy)}=\frac{Q(y, dx)}{\mathsf{m}_0(dx)}. 
\end{equation*}
The following theorem is a restatement of  \cite[Theorem 2.2.2]{CF}.
\begin{thm}[\cite{CF}]\label{DF-pure-jump}
Given a speed function $\lambda >0$. Suppose \eqref{symmetrizing-meas} holds, then the reversible pure jump process $\mathsf{X}$ described above can be characterized by the following Dirichlet form $(\mathfrak{E}, \mathfrak{F})$ on $L^2(\mathsf{E}, \mathsf{m})$ where the underlying reference measure is $\mathsf{m}(dx)=\lambda(x)^{-1}\mathsf{m}_0(dx)$ and
\begin{equation}\label{DF-EN}
\left\{
    \begin{aligned}
        &\mathfrak{F}=  L^2(\mathsf{E},\; \mathsf{m}(x)), \\
        &\mathfrak{E}(f,g) = \frac{1}{2} \int_{\mathsf{E}\times \mathsf{E}} (f(x)-f(y))(g(x)-g(y))Q(x, dy)\mathsf{m}_0(dx).
    \end{aligned}
\right.
\end{equation}
\end{thm}

\subsection{Continuous-time random walks on lattices with varying dimension} \label{S:2.2}

The following proposition is a restatement of \cite[Proposition 2.2]{Lou1}  which gives a description to the behavior  of $X^k$ in the unbounded space with varying dimension $E^k$.

\begin{prop}[\cite{Lou1}]\label{jump-distribution-Xk}
For every $k=1,2,\dots$, $X^k$ has  constant  speed function $\lambda_k=2^{2k}$ and  road map
\begin{equation*}
J_k(x, dy)=\sum_{z\in E^k, \;z\leftrightarrow x \text{ in }G^k}  j_k(x,z)\delta_{\{z\}}(dy), 
\end{equation*}
where
\begin{description}
\item{(i)} 
\begin{equation}\label{original-p_k(x,y)-1}
j_k(x, y)= \frac{1}{v_k(x)}, \quad \text{if }x\in D^k_\eps\cup 2^{-k}\IZ_+, \,y\leftrightarrow x  \text{ in }G^k
\end{equation}
\item{(ii)}
\begin{equation}\label{original-p_k(x,y)-2}
j(a^*_k, y)=\left\{
    \begin{aligned}
         &\frac{1}{v_k(a^*_k)+2^{k+1}-1},  &y \in D^k_\eps,\, y\leftrightarrow x \text{ in }G^k;\\
        &\frac{2^{k+1}}{v_k(a_k^*)+2^{k+1}-1}, &y \in 2^{-k}\IZ_+,\, y\leftrightarrow x \text{ in }G^k.
    \end{aligned}
\right.
\end{equation}
\end{description}
\end{prop}
The next proposition is essentially contained in the proof of \cite[Proposition 2.1]{Lou1}.

\begin{prop}\label{P:2.3}
For any  fixed $0<\eps\le 1/64$ and all $k\ge k_0$, where $k_0$ is specified in \eqref{def-k0},
\begin{equation}\label{mass-ak-star}
m_k(a^*_k)< 2^{-k}.
\end{equation}
\end{prop}
\begin{proof}
It has been shown in the proof of \cite[Proposition 2.1]{Lou1} that for all $k=1,2,\dots$,
\begin{equation} \label{P2.3-1}
v_k(a^*_k)\le 56\eps\cdot 2^k+28.
\end{equation}
It thus follows from \eqref{def-k0} that for $k\ge k_0$,
\begin{align*}
m_k(a^*_k)&\le \frac{2^{-k}}{2}+\frac{2^{-2k}}{4}\left(56\eps \cdot 2^k+28\right)=\frac{2^{-k}}{2}+14\eps\cdot 2^{-k}+7\cdot 2^{-k}\cdot 2^{-k}
\\
&\le \frac{2^{-k}}{2}+14\eps\cdot 2^{-k}+\frac{7\eps}{4}\cdot 2^{-k} \le \left(\frac{1}{2}+\frac{63\eps}{4}\right)2^{-k}< 2^{-k}.
\end{align*}
\end{proof}

\subsection{Isoperimetric inequalities for weighted graphs} \label{S:2.3}

Before we establish Nash-type inequality for $X^k$, we give a summary on the isoperimetric inequalities for weighted graphs. Most of the results in this section can be found in  \cite{MB}. In the following, $\Gamma$ is a locally finite connected graph, and the collection of vertices of $\Gamma$ is denoted by $\mathbb{V}$. If two vertices $x,y\in \mathbb{V}$ are adjacent to each other, then the the unoriented edge connecting $x$ and $y$ is assigned a unique weight $\mu_{xy}>0$.   Set $\mu_{xy}=0$  if $x$ and $y$ are not adjacent in  $\Gamma$.   Denote by $\mu:=\{\mu_{xy}: x,y \text{ connected in }\Gamma\}$ the assignment of the weights on all the unoriented edges. $(\Gamma, \mu)$ is called a locally finite connected {\it weighted graph}. We equip the weighted  graph with $(\Gamma, \mu)$ following measure $\nu$ on $\mathbb{V}$:
\begin{equation}\label{meas-weighted-graph}
\nu (x):= \sum_{y\in \mathbb{V}: y\leftrightarrow x \text{ in }\Gamma} \mu_{xy}, \quad x\in \mathbb{V}.
\end{equation}
 Given two sets of vertices $A, B$ in $\mathbb{V}$, we define
\begin{equation}\label{def-mu-E}
\mu_\Gamma (A, B):=\sum_{x\in A}\sum_{y\in B}\mu_{xy}.
\end{equation}
The following definition of isoperimetric inequality is taken from \cite[Definition 3.1]{MB}.
\begin{dfn}
For $\alpha\in [1, \infty)$, we say that $(\Gamma, \mu)$ satisfies $\alpha$-isoperimetric inequality if there exists $C_0>0$ such that
\begin{equation*}
\frac{\mu_\Gamma(A, \mathbb{V}\backslash A)}{\nu(A)^{1-1/\alpha}} \ge C_0, \quad \text{for every finite non-empty }A\subset \mathbb{V}.
\end{equation*}
\end{dfn}

\begin{prop}[\cite{MB}]\label{MB-iso-implies-nash}
Let $(\Gamma, \mu)$ be a locally finite connected weighted graph satisfying $\alpha$-isoperimetric inequality  with constant $C_0$. Let $\nu$ be the measure defined in \eqref{meas-weighted-graph}. Then $(\Gamma, \mu)$ satisfies the following Nash-type inequality: 
\begin{equation*}
\frac{1}{2}\sum_{x\in \mathbb{V}}\sum_{y\in \mathbb{V}, y\leftrightarrow x} \left(f(x)-f(y)\right)^2\mu_{xy} \ge 4^{-(2+\alpha/2)}C_0^2 \|f\|_{L^2(\nu)}^{2+4/\alpha}\|f\|_{L^1(\nu)}^{-4/\alpha}, \quad f\in L^1(\nu)\cap L^2(\nu)
\end{equation*}
\end{prop}
\begin{proof}
This can be seen combining \cite[Theorem 3.7, Lemma 3.9, Theorem 3.14]{MB} and the proofs therein.
\end{proof}

The next proposition follows immediately from \cite{MB}. As a notation in \cite{MB}, given a weighted graph $(\Gamma, \mu)$ with collection of vertices $\mathbb{V}$. 
 We denote the  counting measure times $2^{-2k}$ on  $\IZ^2$ by  $\mu^{(2)}_k$.
\begin{prop}[\cite{MB}]\label{iso-Z^2}
Let $k\in \mathbb{N}$.  Let all edges of $2^{-k}\IZ_2$ be assigned with a weight of $2^{-2k}/4$.    There exists a constant $K>0$ independent of $k$ such that  for any finite subset $A$ of $2^{-k}\IZ_2$, 
\begin{equation}\label{eq-iso-Z^2}
\mu_{2^{-k}\IZ^2}(A, 2^{-k}\IZ\backslash A)\ge K\cdot 2^{-k}  \mu_k^{(2)} (A)^{1/2}.
\end{equation}
\end{prop}

\section{Nash-type inequalities and equicontinuities for random walks on lattices with varying dimension} \label{S:3}

In the following, we let all the edges in the $2^{-k}\IZ_+\cup\{a^*_k\}$ be assigned with a weight $2^{-k}/2$. Then we define a measure on $2^{-k}\IZ_+\cup\{a^*_k\}$:
\begin{equation}\label{nu1}
\nu^{(1)}_k(x):=\frac{2^{-k}}{2}\cdot \#\{y\in 2^{-k}\IZ_+\cup \{a^*_k\}: y\leftrightarrow x\text{ in } 2^{-k}\IZ_+\cup \{a^*_k\}  \}.
\end{equation}
Similarly, let all the edges in the connected graph $D_\eps^k\cup\{a^*_k\}$ be assigned with a weight of $2^{-2k}/4$. Then we define a measure on 
$D_\eps^k\cup\{a^*_k\}$:
\begin{equation}\label{nu2}
\nu^{(2)}_k(x):=\frac{2^{-2k}}{4}\cdot \#\{y\in D_\eps\cup \{a^*_k\}: y\leftrightarrow x\text{ in } D_\eps\cup \{a^*_k\}  \}.
\end{equation}
\begin{lem}\label{L:3.1}
Let $D_\eps^k\cup \{a^*_k\}$ and $2^{-k}\IZ_+\cup \{a^*_k\} $ be equipped with the weights and measure described in the preceding paragraph. There exists a constant $C>0$ independent of $k$ such that for all $k\ge k_0$,
\begin{equation}\label{iso1}
\frac{\mu_{D_\eps\cup \{a^*_k\}}(A, (D^k_\eps \cup \{a^*_k\})\backslash A)}{\nu^{(2)}_k(A)^{1/2}}\ge 2^{-k}C,\quad \text{for any finite set }A\subset (D^k_\eps \cup \{a^*_k\}),
\end{equation}
and 
\begin{equation}\label{iso2}
\mu_{2^{-k}\IZ_+ \cup \{a^*_k\}}(A, (2^{-k}\IZ_+ \cup \{a^*_k\})\backslash A)  \ge 2^{-k}C,\quad \text{for any finite set }A\subset (2^{-k}\IZ_+ \cup \{a^*_k\}).
\end{equation}
\end{lem}

\begin{proof}
\eqref{iso2} is clear.  In the following we prove \eqref{iso1}. We divide it into two cases depending whether $a^*_k\in A$ or not. For the remainder of this proof, we denote by 
\begin{equation*}
B_\eps^k:=\{x\in 2^{-k}\IZ_2:\, \text{the Euclidean distance between } x \text{ and the origin of }\IR^2 \text{ is  }\le \eps.\}
\end{equation*}
{\it Case (i).} $a^*_k\notin A$. 
\begin{eqnarray}
&& \mu_{D_\eps\cup \{a^*_k\}}(A, \,(D^k_\eps \cup \{a^*_k\})\backslash A) \nonumber
\\
&=& \frac{2^{-2k}}{4} \sum_{x\in A}\#\bigg\{y\in D_\eps^k\backslash A: y\leftrightarrow x \text{ in }D^k_\eps\cup \{a^*_k\}  \bigg \} + \frac{2^{-2k}}{4} \#\bigg\{x\in A: x\leftrightarrow a^*_k \text{ in }D^k_\eps\cup \{a^*_k\}\bigg\} \nonumber 
\\
&\ge & \frac{2^{-2k}}{4} \sum_{x\in A}\#\bigg\{y\in D_\eps^k\backslash A: y\leftrightarrow x \text{ in }2^{-k}\IZ_2 \bigg \} + \frac{2^{-2k}}{4}\cdot  \frac{1}{2} \sum_{x\in A}\#\bigg\{y\in B_\eps^k, y \leftrightarrow x \text{ in }2^{-k}\IZ_2\bigg\}\nonumber
\\
&\ge &  \frac{1}{2}\cdot \frac{2^{-2k}}{4} \sum_{x\in A}\#\bigg\{ y\in 2^{-k}\IZ_2: y\leftrightarrow x \text{ in }2^{-k}\IZ_2\bigg\} \nonumber
\\
&\stackrel{\eqref{def-mu-E}}{=}& \frac{1}{2}\;\mu_{2^{-k}\IZ_2}(A,\, 2^{-k}\IZ_2\backslash A)  \nonumber
\\
&\stackrel{\eqref{eq-iso-Z^2}}{\ge} &\frac{K}{2}\cdot 2^{-k}\mu_k^{(2)}(A)^{1/2} =\frac{K}{2}\cdot 2^{-k}\nu_k^{(2)}(A)^{1/2}, \label{compute-L3.1-1}
\end{eqnarray}
where the first inequality is due to the fact that every edge in $D_\eps^k\cup \{a^*_k\}$   with one endpoint being $a^*_k$ can be identified with at least one, but no more than two edges in $2^{-k}\IZ_2$ that has one end point in $B_\eps^k$. 
\\
{\it Case (ii).} $a^*_k\in A$. For this case,
\begin{eqnarray}
&& \mu_{D_\eps\cup \{a^*_k\}}(A, \,(D^k_\eps \cup \{a^*_k\})\backslash A) \nonumber
\\
&=& \frac{2^{-2k}}{4} \sum_{x\in A\backslash \{a^*_k\}} \#\bigg\{   y\in D_\eps^k\backslash (A\backslash \{a^*_k\}): y\leftrightarrow x     \text{ in }D^k_\eps\cup \{a^*_k\} \bigg\} \nonumber
\\
&+& \frac{2^{-2k}}{4}\#\bigg\{    y \in D^k_\eps\backslash (A\backslash  \{a^*_k\}): y\leftrightarrow a^*_k \text{ in }D^k_\eps\cup \{a^*_k\} \bigg\} \nonumber
\\
&\ge & \frac{2^{-2k}}{4} \sum_{y\in D^k_\eps\backslash (A\backslash  \{a^*_k\})} \#\bigg\{ x\in A\backslash \{a^*_k\}: x\leftrightarrow y \text{  in } 2^{-k}\IZ_2   \bigg\} \nonumber
\\
&+& \frac{1}{2}\cdot \frac{2^{-2k}}{4} \sum_{y \in D^k_\eps\backslash (A\backslash  \{a^*_k\})} \#\bigg\{ x\in B_\eps^k: x\leftrightarrow y \text{ in }2^{-k}\IZ_2\bigg\} \nonumber
\\
&\ge & \frac{1}{2} \cdot \frac{2^{-2k}}{4} \sum_{y \in D^k_\eps\backslash (A\backslash  \{a^*_k\})}\#\bigg\{ x\in (A\backslash \{a^*_k\})\cup B_\eps^k: x\leftrightarrow y \text{ in }2^{-k}\IZ_2\bigg\}    \nonumber
\\
&\stackrel{\eqref{iso-Z^2}}{\ge}  &   \frac{1}{2} \mu_{2^{-k}\IZ_2}((A\backslash \{a^*_k\}),\, 2^{-k}\IZ_2 \backslash A  )  \nonumber
\\
&\ge &\frac{K}{2}\cdot 2^{-k}  \mu_k^{(2)}((A\backslash \{a^*_k\})\cup B_\eps^k)^{1/2}\nonumber
\\
&\stackrel{\eqref{def-mk}}{\ge}  & \frac{K}{2}\cdot 2^{-k} \left( \mu_k^{(2)}((A\backslash \{a^*_k\} )+m_k(a^*_k)-\frac{2^{-k}}{2}    \right)^{1/2} =\frac{K}{2}\cdot 2^{-k}\nu_k^{(2)}(A), \label{compute-L3.1-2}
\end{eqnarray}
where the first inequality above also uses the fact that every edge in $D_\eps^k\cup \{a^*_k\}$   with one endpoint being $a^*_k$ can be identified with at least one, but no more than two edges in $2^{-k}\IZ_2$ that has one end point in $B_\eps^k$. Combining the \eqref{compute-L3.1-1} and \eqref{compute-L3.1-2}, \eqref{iso1} is verified. The proof is thus complete.
\end{proof}

\begin{prop}\label{unif-nash-1}
For every $k\in \mathbb{N}$, let $(P_t^k)_{t\ge 0}$  be  the transition semigroup of $X^k$ with respect to $m_k$.  There exists a constant $C_1>0$ independent of $k$   such that for all $k\ge k_0$, 
\begin{equation}
\|P^k_t\|_{1\rightarrow \infty}\le C_1\left(\frac{1}{t}+\frac{1}{\sqrt{t}}\right), \quad \forall t\in (0, +\infty].
\end{equation}
\end{prop}
\begin{proof}
We  still consider the weighted graph $2^{-k}\IZ_+\cup\{a^*_k\}$ with all edges having equal  weight $2^{-k}/2$, and all edges of the weighted graph $D_\eps^k\cup\{a^*_k\}$ have  an equal weight of $2^{-2k}/4$. Also for now, we let the measures on $2^{-k}\IZ_+\cup\{a^*_k\}$ and $D_\eps^k\cup\{a^*_k\}$ be $\nu^{(1)}_k$ and $\nu^{(2)}_k$, respectively. Lemma \ref{L:3.1} says that the weighted graph $2^{-k}\IZ_+\cup\{a^*_k\}$ described above satisfies isoperimetric inequality $I_1$ with constant $2^{-k}C$, and the weighted graph $D_\eps^k\cup\{a^*_k\}$ described above satisfies isoperimetric inequality $I_2$ also with constant $2^{-k}C$, for all $k\ge k_0$. Therefore, by Proposition \ref{MB-iso-implies-nash}, it holds for all $f\in L^1\left(2^{-k}\IZ_+\cup\{a^*_k\}, \;\nu^{(1)}_k\right) \cap L^2\left(2^{-k}\IZ_+\cup\{a^*_k\}, \;\nu^{(1)}_k\right) $ that
\begin{equation*}
\frac{2^{-k}}{4}\sum_{x\in 2^{-k}\IZ_+\cup \{0\}} \sum_{\substack{y\in 2^{-k}\IZ_+\cup \{0\},\\ y\leftrightarrow x }} \left(f(x)-f(y)\right)^2\ge 4^{-4} 2^{-2k}C^2\|f\|^6_{L^2\left(2^{-k}\IZ_+\cup \{a^*_k\}, \;\nu^{(1)}_k\right)}\cdot\|f\|^{-4}_{L^1\left(2^{-k}\IZ_+\cup\{a^*_k\}, \;\nu^{(1)}_k\right)}.
\end{equation*}
Similarly, also by Proposition \ref{MB-iso-implies-nash}, for all $f\in L^1\left(D_\eps^k\cup\{a^*_k\}, \;\nu^{(2)}_k\right) \cap L^2\left(D_\eps^k\cup\{a^*_k\}, \;\nu^{(2)}_k\right) $,
\begin{equation*}
\frac{2^{-2k}}{8}\sum_{x\in D_\eps^k\cup\{a^*_k\}} \sum_{\substack{y\in D_\eps^k\cup \{a^*_k\},\\ y\leftrightarrow x} } \left(f(x)-f(y)\right)^2\ge 4^{-4} 2^{-2k}C^2  \|f\|^4_{L^2\left(D_\eps^k\cup \{a^*_k\}, \;\nu^{(2)}_k\right)}\cdot\|f\|^{-2}_{L^1\left(D_\eps^k\cup \{a^*_k\}, \;\nu^{(2)}_k\right)}.
\end{equation*}
The above two inequalities can be rewritten as 
\begin{eqnarray}
&&\left( \frac{2^k}{4}\sum_{x\in 2^{-k}\IZ_+\cup\{a^*_k\}} \sum_{\substack{y\in 2^{-k}\IZ_+\cup\{a^*_k\},\\ y\leftrightarrow x} } \left(f(x)-f(y)\right)^2\right)^{1/3}\|f\|^{4/3}_{L^1\left(2^{-k}\IZ_+\cup\{a^*_k\}, \nu^{(1)}_k\right)} \nonumber
\\
&\ge & 4^{-4/3}C^{2/3} \|f\|^2_{L^2\left(2^{-k}\IZ_+\cup\{a^*_k\}, \nu^{(1)}_k\right)} \label{Z1-scale-3}
\end{eqnarray}
and 
\begin{align}\label{Z2-scale-4}
\left( \sum_{x\in D^k_\eps\cup \{a^*_k\}} \sum_{\substack{y\in D^k_\eps\cup \{a^*_k\},\\ y\leftrightarrow x} } \left(f(x)-f(y)\right)^2\right)^{1/2}\cdot\|f\|_{L^1\left(D^k_\eps\cup \{a^*_k\}, \;\nu^{(2)}_k\right)} \ge 4^{-2}C \|f\|^2_{L^2\left(D^k_\eps\cup \{a^*_k\}, \;\nu^{(2)}_k\right)}.
\end{align}
Notice that any $f\in L^1(E^k, m_k)\cap L^2(E^k, m_k)$, it holds that 
\begin{equation*}
\|f\|_{L^p(E^k, m_k)}=\left\|f|_{2^{-k}\cup \{a^*_k\}}\right\|_{L^p\left(2^{-k}\cup \{a^*_k\}, \nu^{(1)}_k\right)}+\left\|f|_{D^k_\eps\cup \{a^*_k\}}\right\|_{L^p\left(D_\eps^k\cup \{a^*_k\}, \nu^{(2)}_k\right)}, \quad p=1,2.
\end{equation*}
For notational convenience, for $f\in L^1(E^k, m_k)\cap L^2(E^k, m_k)$, we set $f_1:=f|_{2^{-k}\cup \{a^*_k\}} $ and $f_2:=f|_{D^k_\eps\cup \{a^*_k\}}$.   Adding up \eqref{Z2-scale-4} and \eqref{Z1-scale-3} yields that for some $c_3>0$ independent of $k$ yields for all $f\in L^1(E^k, m_k)\cap L^2(E^k, m_k)$ that
\begin{eqnarray} 
&&\left(4^{-4/3}C^{2/3}\wedge 4^{-2}C \right)\|f\|^2_{L^2(E^k, m_k)}  \nonumber
\\
&\le &
\left(\frac{2^k}{4}\sum_{x\in 2^{-k}\IZ_+\cup \{0\}} \sum_{\substack{y\in 2^{-k}\IZ_+\cup \{0\},\\ y\leftrightarrow x }} \left(f_1(x)-f_1(y)\right)^2\right)^{1/3} \cdot \|f_1\|_{L^1\left(2^{-k}\IZ_+\cup \{0\}, \;\mu^{(1)}_k\right)}^{4/3}  \nonumber
\\
&+&\left( \sum_{x\in D^k_\eps\cup \{a^*_k\}} \sum_{\substack{y\in D^k_\eps\cup \{a^*_k\},\\ y\leftrightarrow x} } \left(f_2(x)-f_2(y)\right)^2\right)^{1/2}\cdot\|f_2\|_{L^1\left(D^k_\eps\cup \{a^*_k\}, \;\nu^{(2)}_k\right)} \nonumber
\\
&\le & \EE^k(f, f)^{1/3}\|f\|^{4/3}_{L^1(E^k, m_k)}+\EE^k(f, f)^{1/2}\|f\|_{L^1(E^k, m_k)},
\end{eqnarray}
The desired conclusion thus follows immediately by selecting $\mu=1$ and $\nu=2$ in \cite[Corollary 2.12]{CKS}.

\end{proof}

With the Nash-type inequality established above, next we use Davies method to get a heat kernel upper bound for $(\EE^k, \mathcal{D}(\EE^k))$ that can be viewed as ``equicontinuity" of the heat kernels of $\{X^k, k\ge 1\}$ in $k$. For this purpose, we first rewrite \eqref{DF-RWVD-form} as follows:
\begin{align}\label{DF-RWVD-rewrite}
\left\{
\begin{aligned}
&\mathcal{D}(\EE^{k})=L^2(E^k, m_k)
\\
&\EE^{k}(f, f)= \frac{2^{2k}}{8}\sum_{x\in D_\eps^k}\sum_{y\leftrightarrow x}\left(f(y)-f(x)\right)^2m_k(x)+\frac{2^{2k}}{4}\sum_{x\in 2^{-k}\IZ_+}\sum_{y\leftrightarrow x}\left(f(y)-f(x)\right)^2m_k(x)
\\
&\qquad \quad\,\,\; +\frac{2^{2k}}{8}\sum_{\substack{y\leftrightarrow a^*_k\\ y\in D_\eps^k}}\left(f(y)-f(a^*_k)\right)^2\frac{2^{-2k}}{4}+\frac{2^{2k}}{4}\sum_{\substack{y\leftrightarrow a^*_k\\ y\in 2^{-k}\IZ_+}}\left(f(y)-f(a^*_k)\right)^2\frac{2^{-k}}{2}.
\end{aligned}
\right.
\end{align}
Now for every $k\in \mathbb{N}$, we define a metric $d_k(\cdot, \cdot)$ on $E^k$ as follows: 
\begin{equation}\label{def-dk}
d_k(x,y):=2^{-k}\times \text{number of edges along the shortest (geodesic) path between }x \text{ and }y \text{ in }G^k.
\end{equation}

In the next proposition as well as the  corollary following it, we establish an heat kernel upper bound estimate for $X^k$ using Davies's method. In the remainder of this paper, for every $k\in \mathbb{N}$,  we denote by $p_k(t,x, y)$ the transition density function of $(P_t^k)_{t\ge 0}$ with respect to $m_k$.

\begin{prop}\label{general-davies}
There exist $C_2>0$ independent of $k$, such that for any $k\ge k_0$ fixed in \eqref{def-k0} and any $\alpha_k\le 2^{k-1}$, 
\begin{equation}
p_k(t,x,y)\le 
C_2\left(\frac{1}{t}\vee \frac{1}{\sqrt{t}}\right)\exp\left(-\alpha_k d_k(x,y)+2t\alpha_k^2\right),\quad 0<t<\infty, \, x,y\in E^k.
\end{equation}
\end{prop}

\begin{proof}
We prove this result using \cite[Corollary 3.28]{CKS}. Towards this, for each $k$, we set $\wh{\FF}^k:=\{h+c:\, h \in \mathcal{D}(\EE^k), \, h\text{ bounded, and }c\in \IR^1\}$. It is known that the regular symmetric Dirichlet form $(\EE^k, \mathcal{D}(\EE^k))$ can  be written in terms of an energy measure $\Gamma^k$ as follows:
$$
\EE^k(u, u)=\int_{E^k}\Gamma^k(u, u), \quad u\in \wh{\FF} ^k.
$$
where $\Gamma$ is a positive semidefinite symmetric bilinear form on $\FF$ with values being signed Radon measures on $E$, which is also called the energy measure.   Now we define $\wh{\FF}^k_\infty$ as a subset of $\psi \in \wh{\FF}^k$ satisfying the following conditions:
\begin{description}
\item{(i)} Both $e^{-2\psi}\Gamma^k(e^{\psi}, e^{\psi})$ and $e^{2\psi}\Gamma^k(e^{-\psi}, e^{-\psi})$  as measures are absolutely continuous with respect to $m_k$ on $E^k$. 
\item{(ii)} Furthermore,
\begin{equation}
\Gamma^k(\psi):=\left( \left\|\frac{de^{-2\psi }\;\Gamma^k(e^{\psi}, e^{\psi})}{dm_k} \right\|_\infty \vee  \left\|\frac{de^{2\psi }\;\Gamma^k(e^{-\psi}, e^{-\psi})}{dm_k} \right\|_\infty \right)^{1/2}<\infty.
\end{equation}
\end{description}
 We fixed  a constant $\alpha_k \le2^{k-1}$ and denote by
\begin{equation}\label{def-psi-kn}
\psi_{k, n}(x):=\alpha_k \cdot \left(d_k(x, a^*_k)\wedge n\right).
\end{equation}
In order to apply \cite[Corollary 3.28]{CKS}, we first check that for every $n$,  $\psi_{k,n}\in \wh{\FF}_\infty^k$. Notice that $\psi_{k,n}$ is a constant outside of a bounded domain of $E^k$, therefore it is in $\wh{\FF}^k$. Now we compute $e^{-2\psi_{k,n}}\Gamma^k(e^{\psi_{k,n}}, e^{\psi_{k,n}})$ as a measure.   Noting that $\psi_{k,n}(a^*_k)=0$, we first rewrite the expression of $\EE^k$ in \eqref{DF-RWVD-form} as follows: for $f\in \wt{\FF}^k_\infty$,
\begin{eqnarray}
&&\int_{E^k}\Gamma^k(f,f)=\EE^k(f, f) \nonumber
\\
&=& \frac{1}{8}\sum_{\substack{e^o_{xy}:\; e_{xy}\in G^k_e,\\ x,y\in D^k_\eps\cup \{a_k^*\} }} \left(f(x)-f(y)\right)^2 +\frac{2^k}{4}\sum_{\substack{e^o_{xy}:\;e_{xy}\in G^k_e,\\ x,y\in 2^{-k}\IZ_+\cup \{a_k^*\} }}\left(f(x)-f(y)\right)^2 \nonumber
\\
&=& \frac{2^{2k}}{8}\sum_{x\in D_\eps^k\cap A}  \sum_{y\leftrightarrow x}\left(f(y)-f(x)\right)^2m_k(x) +\frac{2^{2k}}{4}\sum_{x\in 2^{-k}\IZ_+\cap A}  \sum_{y\leftrightarrow x}\left(f(y) -f(x))\right)^2m_k(x) \nonumber
\\
&+&\mathbf{1}_{\{a^*_k\in A\}}\left(\frac{2^{2k}}{8}\sum_{\substack{y\leftrightarrow a^*_k\\ y\in D_\eps^k}}\left(f(y)-f(a^*_k)\right)^2\frac{2^{-2k}}{4}+\frac{2^{2k}}{4}\sum_{\substack{y\leftrightarrow a^*_k\\ y\in 2^{-k}\IZ_+}}\left(f(y)-f(a^*_k)\right)^2\frac{2^{-k}}{2}\right). \label{Prop3.2-rewrite-EEk}
\end{eqnarray}
On account of \eqref{Prop3.2-rewrite-EEk}, given any subset $A\subset E^k$,   we have
\begin{eqnarray}
&&  e^{-2\psi_{k,n}}\Gamma^k(e^{\psi_{k,n}}, e^{\psi_{k,n}}) (A) \nonumber
\\
&=&\frac{2^{2k}}{8}\sum_{x\in D_\eps^k\cap A} e^{-2\psi_{k,n}(x)}   \sum_{y\leftrightarrow x}\left(e^{\psi_{k,n}(y)}-e^{\psi_{k,n}(x)}\right)^2m_k(x) \nonumber
\\
&+&\frac{2^{2k}}{4}\sum_{x\in 2^{-k}\IZ_+\cap A} e^{-2\psi_{k,n}(x)}    \sum_{y\leftrightarrow x}\left(e^{\psi_{k,n}(y)}  -e^{\psi_{k,n}(x)} )\right)^2m_k(x) \nonumber
\\
&+&\mathbf{1}_{\{a^*_k\in A\}}\left(\frac{2^{2k}}{8}\sum_{\substack{y\leftrightarrow a^*_k\\ y\in D_\eps^k}}\left(e^{\psi_{k,n}(y)}-1\right)^2\frac{2^{-2k}}{4}+\frac{2^{2k}}{4}\sum_{\substack{y\leftrightarrow a^*_k\\ y\in 2^{-k}\IZ_+}}\left(e^{\psi_{k,n}(y)}-1\right)^2\frac{2^{-k}}{2}\right). \nonumber
\end{eqnarray}
In view of the definition of $m_k$ in \eqref{def-mk}, we further have
\begin{eqnarray}
&&  e^{-2\psi_{k,n}}\Gamma^k(e^{\psi_{k,n}}, e^{\psi_{k,n}}) (A) \nonumber
\\
&\le &\frac{1}{8}\sum_{x\in D_\eps^k\cap A} \sum_{y\leftrightarrow x} \left( e^{\alpha_k\left(d_k(y, a^*_k)\wedge n- d_k(x,a^*_k)\wedge n\right)}-1\right)^2 \nonumber
\\
&+&\frac{2^k}{4}\sum_{x\in 2^{-k}\IZ\cap A} \sum_{y\leftrightarrow x} \left( e^{\alpha_k\left(d_k(y, a^*_k)\wedge n- d_k(x,a^*_k)\wedge n\right)}-1\right)^2\nonumber
\\
&+&\mathbf{1}_{\{a^*_k\in A\}}\left(\frac{1}{32}\sum_{\substack{y\leftrightarrow a^*_k\\ y\in D_\eps^k}} \left(e^{\alpha_k\left(d_k(y,a^*_k)\wedge n\right)}-1\right)^2+\frac{2^{k}}{8}\sum_{\substack{y\leftrightarrow a^*_k\\ y\in 2^{-k}\IZ_+}} \left(e^{\alpha_k\left(d_k(y,a^*_k)\wedge n \right)}-1\right)^2\right). \label{computation-Gammak-1}
\end{eqnarray}
From the above it is clear that $e^{-2\psi_{k,n}}\Gamma^k(e^{\psi_{k,n}}, e^{\psi_{k,n}})$ is absolutely continuous with respect to $m_k$. By similar compution, one can show that so is $e^{2\psi_{k,n}}\Gamma^k(e^{-\psi_{k,n}}, e^{-\psi_{k,n}})$. Thus   $\psi_{k,n}$ satisifies the first condition in the definition of $\wh{\FF}^k_\infty$.

Next we verify that $\psi_{k,n}$ also satisfies the second condition of the definition of $\wh{\FF}^k_\infty$.  It is obvious from \eqref{def-mk} that
\begin{equation}\label{mk-ineq}
m_k(x)\ge \left\{
    \begin{aligned}
         &\frac{2^{-2k}}{4},  &x \in D^k_\eps;\\
        &\frac{2^{-k}}{2}, &x \in 2^{-k}\IZ_+\cup \{a^*_k\}.\\
    \end{aligned}
\right.
\end{equation}
Hence combining \eqref{computation-Gammak-1} and \eqref{mk-ineq}  shows that
\begin{eqnarray}
&&\left\|\frac{de^{-2\psi_{k,n} }\;\Gamma^k(e^{\psi_{k,n}}, e^{\psi_{k,n}})}{dm_k} \right\|_\infty  \nonumber
\\
&\le & \max_{x\in D_\eps^k } \left\{\frac{2^{2k}}{2}\sum_{y\leftrightarrow x}\left( e^{\alpha_k\left(d_k(y, a^*_k)
\wedge n- d_k(x,a^*_k)\wedge n\right)}-1\right)^2 \right\} \nonumber
\\
&\vee & \max_{x\in 2^{-k}\IZ} \left\{  \frac{2^{2k}}{2}\sum_{y\leftrightarrow x}\left( e^{\alpha_k\left(d_k(y, a^*_k)\wedge n- d_k(x,a^*_k)\wedge n\right)}-1\right)^2\right\} \nonumber
\\
&\vee & \left(\frac{2^k}{16}\sum_{\substack{y\leftrightarrow a^*_k\\ y\in D_\eps^k}} \left(e^{\alpha_k\left(d_k(y,a^*_k)\wedge n\right)}-1\right)^2+\frac{2^{2k}}{4}\sum_{\substack{y\leftrightarrow a^*_k\\ y\in 2^{-k}\IZ_+}} \left(e^{\alpha_k\left(d_k(y,a^*_k)\wedge n \right)}-1\right)^2\right). \label{Radon-nikodyn-1}
\end{eqnarray}
By similar computation one can show that
\begin{eqnarray}
&&\left\|\frac{de^{2\psi_{k,n} }\;\Gamma^k(e^{-\psi_{k,n}}, e^{-\psi_{k,n}})}{dm_k} \right\|_\infty \nonumber
\\
&\le & \max_{x\in D_\eps^k } \left\{\frac{2^{2k}}{2}\sum_{y\leftrightarrow x}\left( e^{\alpha_k\left(d_k(x, a^*_k)
\wedge n- d_k(y,a^*_k)\wedge n\right)}-1\right)^2 \right\} \nonumber
\\
&\vee & \max_{x\in 2^{-k}\IZ} \left\{  \frac{2^{2k}}{2}\sum_{y\leftrightarrow x}\left( e^{\alpha_k\left(d_k(x, a^*_k)\wedge n- d_k(y,a^*_k)\wedge n\right)}-1\right)^2\right\} \nonumber
\\
&\vee & \left(\frac{2^k}{16}\sum_{\substack{y\leftrightarrow a^*_k\\ y\in D_\eps^k}} \left(e^{-\alpha_k\left(d_k(y,a^*_k)\wedge n\right)}-1\right)^2+\frac{2^{2k}}{4}\sum_{\substack{y\leftrightarrow a^*_k\\ y\in 2^{-k}\IZ_+}} \left(e^{-\alpha_k\left(d_k(y,a^*_k)\wedge n \right)}-1\right)^2\right). \label{Radon-nikodyn-2}
\end{eqnarray}
Combining \eqref{Radon-nikodyn-1} and \eqref{Radon-nikodyn-2}  we get
\begin{align}
\Gamma^k(\psi_{k,n})&=\left( \left\|\frac{de^{-2\psi_{k,n} }\;\Gamma^k(e^{\psi_{k,n}}, e^{\psi_{k,n}})}{dm_k} \right\|_\infty \vee  \left\|\frac{de^{2\psi_{k,n} }\;\Gamma^k(e^{-\psi_{k,n}}, e^{-\psi_{k,n}})}{dm_k} \right\|_\infty \right)^{1/2} \nonumber
\\
&\le \left(\max_{x\in D_\eps^k } \left\{\frac{2^{2k}}{2}\sum_{y\leftrightarrow x}\left( e^{\alpha_k\left|d_k(x, a^*_k)
\wedge n- d_k(y,a^*_k)\wedge n\right|}-1\right)^2 \right\}\right)^{1/2} \nonumber
\\
& \vee \left( \max_{x\in 2^{-k}\IZ} \left\{  \frac{2^{2k}}{2}\sum_{y\leftrightarrow x}\left( e^{\alpha_k\left|d_k(x, a^*_k)\wedge n- d_k(y,a^*_k)\wedge n\right|}-1\right)^2\right\} \right)^{1/2}\nonumber
\\
&\vee  \left(\frac{2^k}{16}\sum_{\substack{y\leftrightarrow a^*_k\\ y\in D_\eps^k}} \left(e^{\alpha_k\left(d_k(y,a^*_k)\wedge n\right)}-1\right)^2+\frac{2^{2k}}{4}\sum_{\substack{y\leftrightarrow a^*_k\\ y\in 2^{-k}\IZ_+}} \left(e^{\alpha_k\left(d_k(y,a^*_k)\wedge n \right)}-1\right)^2\right)^{1/2}. \label{computation-psi-1}
\end{align}
In the following we claim that   all  three ``$(\cdots)^{1/2}$" terms on the right hand side of \eqref{computation-psi-1}   are bounded by $\sqrt{2}\alpha_k$, where $\alpha_k$ is fixed in the line above  \eqref{def-psi-kn}.  First of all, we  note that $e^x-1<2x$, for all $0<x<1/2$. Thus for all $k\in \mathbb{N}$,
\begin{equation}\label{alpha-computation-1}
e^{\alpha_k x}-1\le 2\alpha_k x,\quad \text{for }0<x\le 2^{-k}, \,\alpha_k\le 2^{k-1}.
\end{equation}
In view of  the definition of $d_k$, $d_k(x,y)\le 2^{-k}$ for any  $x\leftrightarrow y$. Thus for all $k\ge k_0\ge 1$,
\begin{eqnarray}
 &&\left(\max_{x\in D_\eps^k } \left\{\frac{2^{2k}}{2}\sum_{y\leftrightarrow x}\left( e^{\alpha_k\left|d_k(x, a^*_k)
\wedge n- d_k(y,a^*_k)\wedge n\right|}-1\right)^2 \right\}\right)^{1/2} \nonumber 
\\
&\le & \frac{2^k}{\sqrt{2}} \max_{x\in D_\eps^k }\left\{ \sum_{y\leftrightarrow x} \left( e^{\alpha_k\left|d_k(x, a^*_k)
\wedge n- d_k(y,a^*_k)\wedge n\right|}-1\right) \right\}\nonumber
\\
\eqref{alpha-computation-1} &\le & \frac{2^k}{\sqrt{2}} \cdot \max_{x\in D_\eps^k } \sum_{y\leftrightarrow x} 2\left(\alpha_k\left|d_k(x, a^*_k)
\wedge n- d_k(y,a^*_k)\wedge n\right|\right) \nonumber
\\
 &\le &   \frac{2^k}{\sqrt{2}}  \cdot 2\alpha_k\cdot 2^{-k} \le \sqrt{2}\alpha_k\label{psi-bound-1}
\end{eqnarray}
By similar computation one can show that for the second ``$(\cdots)^{1/2}$",
\begin{equation}\label{psi-bound-2}
\left( \max_{x\in 2^{-k}\IZ} \left\{  \frac{2^{2k}}{2}\sum_{y\leftrightarrow x}\left( e^{\alpha_k\left|d_k(x, a^*_k)\wedge n- d_k(y,a^*_k)\wedge n\right|}-1\right)^2\right\} \right)^{1/2} \le \sqrt{2}\alpha_k.
\end{equation}
Finally, to bound the last ``$(\dots)^{1/2}$" on the right hand side of \eqref{computation-psi-1}, we have for  $0<\eps<1/64$ and $k\ge k_0$ that
\begin{eqnarray}
&& \left(\frac{2^k}{16}\sum_{\substack{y\leftrightarrow a^*_k\\ y\in D_\eps^k}} \left(e^{\alpha_k\left(d_k(y,a^*_k)\wedge n\right)}-1\right)^2+\frac{2^{2k}}{4}\sum_{\substack{y\leftrightarrow a^*_k\\ y\in 2^{-k}\IZ_+}} \left(e^{\alpha_k\left(d_k(y,a^*_k)\wedge n \right)}-1\right)^2\right)^{1/2}\nonumber
\\
(\text{Proposition }\ref{P:2.3})&\le & \left(\frac{2^{k}}{16}\cdot \left(56\eps\cdot 2^{k}+28\right)\cdot \max_{\substack{y\leftrightarrow a^*_k\\ y\in D_\eps^k}}\left( e^{\alpha_k\left(d_k(y,a^*_k)\wedge n\right)}-1\right)^2+\frac{2^{2k}}{4}\cdot \left(e^{\alpha_k\cdot 2^{-k}}-1\right)^2\right)^{1/2} \nonumber
\\
&\le &\left( \left(\frac{7\eps}{2}+\frac{7\cdot 2^{-k}}{4}\right)\cdot 2^{2k}\cdot \left(e^{\alpha_k\cdot 2^{-k}}-1\right)^2+\frac{2^{2k}}{4}\left(e^{\alpha_k\cdot 2^{-k}}-1\right)^2\right)^{1/2} \nonumber
\\
\eqref{def-k0}&< & \left(\frac{2^{2k}}{2}\left(e^{\alpha_k\cdot 2^{-k}}-1\right)^2\right)^{1/2}\le \frac{2^k}{\sqrt{2}}\left(e^{\alpha_k\cdot 2^{-k}}-1\right)\stackrel{\eqref{alpha-computation-1}}{\le} \sqrt{2}\alpha_k.\label{psi-bound-3}
\end{eqnarray}
Combining \eqref{psi-bound-1}, \eqref{psi-bound-2}, and \eqref{psi-bound-3}, we have for $k\ge k_0$,
\begin{equation}
\Gamma^k(\psi_{k,n})\le \sqrt{2}\alpha_k.
\end{equation}
This also shows that $\psi_{k,n}$ is indeed in the class $\wh{\FF}_\infty$ defined in the second paragraph of the proof. Finally, to  complete using \cite[Corollary 3.28]{CKS},  by Proposition \ref{unif-nash-1} and the choice of $\psi_{k,n}$ in \eqref{def-psi-kn}, we have that for every $x\in E^k$ and every $n$, there exists a constant $C_2>0$ only depending through the $C_1$  given  in Proposition \ref{unif-nash-1}  but not $k$ such that
\begin{equation*}
p_k(t,x,y)\le 
C_2\left(\frac{1}{t}\vee \frac{1}{\sqrt{t}}\right)  \exp\left(-\alpha_k \left(d_k(x,y)\wedge n\right)+2t\alpha_k^2\right), \quad \text{for }0<t<\infty, m_k\text{-a.e.  }y\in E^k.
\end{equation*}
Since all singletons in $E^k$ have strictly positive measures, we may drop the ``a.e." in the inequality above. Finally, the proof is completed by letting $n\rightarrow \infty$.
\end{proof}

\begin{co}\label{HKUB}
There exist $C_3>0$ independent of $k$, such that for any $k\ge k_0$ fixed in \eqref{def-k0} and any $\alpha_k\le 2^{k-1}$, it holds for all $x,y\in E^k$ and all $t\ge 0$ that
\begin{equation}
p_k(t,x,y)\le \left\{
\begin{aligned}
&C_3\left(\frac{1}{t}\vee \frac{1}{\sqrt{t}}\right)e^{-d_k(x,y)^2/(32t)}, & \text{ when }d_k(x,y)\le 16\cdot 2^kt;
\\
& C_3\left( \frac{1}{t}\vee \frac{1}{\sqrt{t}} \right) e^{-2^k d_k(x,y)/2},
& \text{ when }d_k(x,y)\ge 16\cdot 2^kt.
\end{aligned}
\right.
\end{equation}
In particular, given any $T>0$, there exists $C_4>0$ such that
\begin{equation}
p_k(t, x, y)\le \frac{C_4}{t}\left( e^{-d_k(x,y)^2/(32t)} + e^{-2^k d_k(x,y)/2} \right), \, \text{for all }(t, x, y)\in (0, T]\times E^k\times E^k.
\end{equation}
\end{co}
\begin{proof}
To prove this, in Proposition \ref{general-davies}, given any $k\ge k_0$, for any fixed $t_0>0$ and any pair of $x_0, y_0\in E^k$, we take 
\begin{equation*}
\alpha_k:=\frac{d_k(x_0,y_0)}{16t_0}\wedge 2^{k}.
\end{equation*}
Then Proposition \ref{general-davies} yields that  for all $t>0$ and $x,y\in E^k$, 
\begin{align*}
p_k(t_0,x, y)\le C_2 \left(\frac{1}{t_0}\vee \frac{1}{\sqrt{t_0}}\right)\exp\left[-\left(\frac{d_k(x_0,y_0)}{16t_0}\wedge 2^{-k}\right) d_k(x,y)+2t_0\left( \frac{d_k(x_0,y_0)}{16t_0}\wedge 2^{-k}   \right)^2\right],
\end{align*}
where $C_2$ only depends on the $C_1$ in Proposition \ref{unif-nash-1}. Taking $x=x_0$ and $y=y_0$ yields that
\begin{align}\label{C:3.3-compute-1}
p_k(t_0,x_0, y_0)\le C_2 \left(\frac{1}{t_0}\vee \frac{1}{\sqrt{t_0}}\right)\exp\left[-\left(\frac{d_k(x_0,y_0)}{16t_0}\wedge 2^{-k}\right) d_k(x_0,y_0)+2t_0\left( \frac{d_k(x_0,y_0)}{16t_0}\wedge 2^{-k}   \right)^2\right]. 
\end{align}
Now we divide our discussion into two cases:
\\
{\it Case 1.} $d_k(x_0, y_0)\ge 16\cdot 2^k t_0$. Then for the exponential term on the right hand side  of  \eqref{C:3.3-compute-1}   it holds
\begin{eqnarray}
&&-\left(\frac{d_k(x_0,y_0)}{16t_0}\wedge 2^{-k}\right) d_k(x_0,y_0)+2t_0\left( \frac{d_k(x_0,y_0)}{16t_0}\wedge 2^{-k}   \right)^2   \nonumber
\\
&\le & -2^k d_k(x_0, y_0)+2t_0 (2^k)^2 \nonumber
\\
&\le & -2^k d_k(x_0, y_0)+2t_0 \frac{d_k(x_0, y_0)}{16t_0}\cdot 2^k \nonumber
\\
&\le &-2^k d_k(x_0, y_0)+\frac{1}{8}\cdot 2^k\cdot d_k(x_0, y_0)\le -\frac{2^k}{2}d_k(x_0, y_0). \label{C:3.3-compute-2}
\end{eqnarray}
{\it Case 2.} $d_k(x_0, y_0)\le 16\cdot 2^k t_0$. For this case, 
\begin{eqnarray}
&&-\left(\frac{d_k(x_0,y_0)}{16t_0}\wedge 2^{-k}\right) d_k(x_0,y_0)+2t_0\left( \frac{d_k(x_0,y_0)}{16t_0}\wedge 2^{-k}   \right)^2   \nonumber
\\
&\le & -\frac{d_k(x_0, y_0)^2}{16t_0} +2t_0\left( \frac{d_k(x_0, y_0)}{16t_0} \right)^2 \le -\frac{d_k(x_0, y_0)^2}{32t_0}. \label{C:3.3-compute-3}
\end{eqnarray}
The proof is completed by substituting the power of the exponential term on the right hand side of \eqref{C:3.3-compute-1} with \eqref{C:3.3-compute-2} and \eqref{C:3.3-compute-3} for the two cases, respectively. 
\end{proof}

\section{Tightness of random walks on spaces with varying dimensions} \label{S:4}

In order to use Corollary \ref{HKUB} to establish the tightness of $\{X^k\}_{k\ge 1}$, we first record a proposition in \cite[Chapter VI Theorem 3.21]{JS}, which provides a criterion for tightness for c\`{a}dl\`{a}g processes  adapted to our setting. As a standard notation, given a metric $d(\cdot,\, \cdot)$, we denote by
\begin{equation*}
w_d(x,\, \theta,\, T):=\inf_{\{t_i\}_{1\le i\le n}\in \Pi} \max_{1\le i\le n} \sup_{s, t\in [t_i, t_{i-1}]} d(x(s), x(t)),
\end{equation*}
where $\Pi$ is the collection of all possible partitions of the form $0=t_0<t_1<\cdots <t_{n-1}<T\le t_n$ with $\min_{1\le i\le n} (t_i-t_{i-1})\ge \theta$ and $n\ge 1$.
\begin{prop}[Chapter VI, Theorem 3.21 in \cite{JS}]\label{tightness-criterion}
Let $\{Y_k, \IP^y\}_{k\ge 1}$ be a a sequence of c\`{a}dl\`{a}g processes on state space $E$.  Given $y\in E$, the laws of $\{Y_k, \IP^y\}_{k\ge 1}$ are tight in the Skorokhod space $D([0, T],E, \rho)$  if and only if
\begin{description}
\item{(i).} For any $T>0$, $\delta>0$, there exist $N_1\in \mathbb{N}$ and $M>0$ such that for all $k\ge N_1$,
\begin{equation}
\IP^y\left[\sup_{t\in [0, T]}\big|Y^k_t\big|_\rho>M\right]<\delta. 
\end{equation} 
\item{(ii).} For any $T>0$, $\delta_1, \delta_2>0$, there exist $\delta_3>0$ and $N_2>0$ such that for all $k\ge N_2$,
\begin{equation}
\IP^y\left[w_\rho\left( Y^k  ,\, \delta_3,\, T\right)>\delta_1\right]<\delta_2. 
\end{equation}
\end{description}
\end{prop}

In order to establish the tightness of $\{X_n\}_{n\ge 1}$ by veritying the two conditions in Proposition \ref{tightness-criterion} using the heat kernel upper bounds in Corollary \ref{HKUB}, we first prepare the following simple lemma which will be used later in this paper. 

\begin{lem}\label{L:strong-markov}
Given any $T, M>0$, for any sufficiently large $k\ge k_0$ specified in \eqref{def-k0} such that $2^{-k}<T$, it holds for all $x\in E^k$  that
\begin{align}
\IP^x\left[\sup_{t\in [0, T]}|X^k_t|_\rho\ge M\right]&\le \IP^x\left[   \sup_{t\in [0, 8^{-k}]} |X^k_t|_\rho\ge M \right]+\IP^x\left[\left|X_T^k\right|_\rho \ge \frac{M}{2}\right] \nonumber
\\
& +\IP^x\left[ T-8^{-k}\le \tau_M  \le T,\, \left|X^k_T\right|_\rho\le \frac{M}{2}\right]  \nonumber
\\
&+\IP^x\left[8^{-k}\le \tau_M\le T-8^{-k},\, \left|X^k_T\right|_\rho\le  \frac{M}{2} \right],  \nonumber
\end{align}
where $\tau_M:=\inf\{t>0: \;|X^k_t|_\rho\ge M\}$.
\end{lem}
\begin{proof}
By inclusions of events, we have
\begin{eqnarray}
&&\IP^x\left[\sup_{t\in [0, T]}|X^k_t|_\rho\ge M\right] \nonumber
\\
&=&\IP^x\left[\sup_{t\in [0, 8^{-k}]}|X^k_t|_\rho\ge M\right] +\IP^x\left[\sup_{t\in [8^{-k}, T]}|X^k_t|_\rho\ge M,\sup_{t\in [0, 8^{-k}]}|X^k_t|_\rho < M \right] \nonumber
\\
&=&\IP^x\left[\sup_{t\in [0, 8^{-k}]}|X^k_t|_\rho\ge M\right] +\IP^x\left[\sup_{t\in [8^{-k}, T]}|X^k_t|_\rho\ge M,\,|X_T^k|_\rho\ge \frac{M}{2},\sup_{t\in [0, 8^{-k}]}|X^k_t|_\rho <M \right] \nonumber
\\
&+& \IP^x\left[\sup_{t\in [8^{-k}, T]}|X^k_t|_\rho\ge M,\, |X_T^k|_\rho\le \frac{M}{2},\sup_{t\in [0, 8^{-k}]}|X^k_t|_\rho < M \right] \nonumber
\\
 &\le & \IP^x\left[\sup_{t\in [0, 8^{-k}]}|X^k_t|_\rho\ge M\right] +\IP^x\left[\left|X_T^k\right|\rho \ge \frac{M}{2}\right] +\IP^x\left[8^{-k}\le \tau_M\le T,\, \left|X^k_T\right|_\rho \le  \frac{M}{2} \right] \nonumber
 \\
&\le & \IP^x\left[\sup_{t\in [0, 8^{-k}]}|X^k_t|_\rho\ge M\right] +\IP^x\left[\left|X_T^k\right|\rho \ge \frac{M}{2}\right] +\IP^x\left[8^{-k}\le \tau_M\le T-8^{-k},\, \left|X^k_T\right|_\rho\le  \frac{M}{2} \right]\nonumber
\\
&+& \IP^x\left[T-8^{-k}\le \tau_M\le T,\, \left|X^k_T\right|_\rho\le \frac{M}{2} \right]. \label{computation-strong-Markov-1}
\end{eqnarray}
\end{proof}

In the next few propositions,  by verifying the two conditions in Proposition \ref{tightness-criterion} respectively, we establish the tightness of the laws of $\{X^k_t, \IP^{a^*_k}\}_{k\ge 1}$. 

\begin{prop}\label{P:3.6}
For any $\delta>0$, any $T>0$, there exists $M_1>0$ such that for all $k\ge k_0$ specified in \eqref{def-k0}: 
\begin{equation}
\sup_{y\in E^k}\IP^y\left[ \sup_{t\in [0, 8^{-k}]} \rho\left(X_0^k, X_t^k\right)\ge M_1\right]<\delta.
\end{equation}
\end{prop}

\begin{proof}
We denote by $T_0^k:=0$, and 
\begin{equation*}
T_l^k := \inf\{t>T_{l-1}^k:\; X^k_{l} \neq X^k_l\},\quad \text{for }l=1,2,\dots,
\end{equation*}
i.e., $T_l^k$ is the $l^{\text{th}}$ holding time of $X^k$, and $\{T^k_l\}_{l\ge 1}$ are i.i.d. exponential random variables, each with mean $2^{-k}$.  We then denote by $N_t^k:=\sup\{l\ge 1:T_l^k\le t\}$.  It is clear that $N_t^k$ follows Poisson distribution with parameter $t\cdot 2^{2k}$  Recall that $d
_k(x,y)$ is the smallest number of edges that a path with endpoints $x$ and $y$ has multiplied by $2^{-k}$, which implies that for any $k\in \mathbb{N}$ and  $x,y\in E^k$, $d_k(x,y)\ge \rho(x,y)$. Therefore, for any $y\in E^k$, given any $\delta>0$, for any $M$ satisfying $\frac{e}{M}<\frac{\delta\wedge 1}{16}$, it holds that
\begin{eqnarray}
\IP^y\left[ \sup_{t\in [0, 8^{-k}]} \rho\left(X_0^k, X_t^k\right)\ge M\right] &\le & \IP^y\left[ \sup_{t\in [0, 8^{-k}]}  d_k\left(y, X_t^k\right)\ge M\right] \nonumber
\\
&\le &  \IP^y\left[2^{-k}N_{8^{-k}}^k\ge M\right] \nonumber
\\
(N_t^k\sim \text{Poisson}(t\cdot 2^{2k}))&=&e^{-8^{-k}\cdot 2^{2k}}\sum_{j=M\cdot 2^k}^\infty \frac{(8^{-k}\cdot 2^{2k})^j}{j!} \nonumber
\\
(\text{Stirling's formula})&\le &   \sum_{j=M\cdot 2^k}^\infty \frac{(8^{-k}\cdot 2^{2k})^j}{j^j e^{-j}\sqrt{2\pi j}} \nonumber
\\
&\le & \sum_{j=M\cdot 2^k}^\infty \left(\frac{8^{-k}4^{k}e}{M\cdot 2^k}  \right)^j \nonumber
\\
&\le & \sum_{j=1}^\infty\left(\frac{e}{M}\right)^j <\frac{\frac{\delta \wedge 1}{16}}{1-\frac{\delta \wedge 1}{16}} \le \delta \wedge 1.
\end{eqnarray}
\end{proof}

\begin{prop}\label{P:3.7}
For any $\delta>0$, any $T\ge 1$, there exists $M_2>0$  such that for all $k\ge k_0$ specified in \eqref{def-k0}: 
\begin{equation*}
\sup_{8^{-k}\le t\le T}\IP^{a^*_k}\left[d_k(X_t^k, a^*_k)\ge M_2\right]<\delta.
\end{equation*}
\end{prop}

\begin{proof}
To prove this, we utilize Proposition \ref{HKUB}. Given any $k\ge k_0 $  fixed  in \eqref{def-k0}, any $T\ge 1$,  and any $t\in [8^{-k}, T]$,
\begin{eqnarray}
\IP^{a^*_k}\left[ d_k(X_t^k, a^*_k)\ge M\right]  \nonumber
&\le& \sum_{d_k(y, a^*)\ge M} \frac{c_1}{t}\left(e^{-\frac{d_k(a^*_k, y)^2}{32t}}+ e^{-\frac{2^kd_k(a^*_k,y)}{2}}\right) m_k(dy). \nonumber
\\
&\le &  \sum_{\substack{y\in 2^{-k}\IZ_+\\d_k(y, a_k^*)\ge M}} \frac{c_1}{t}\left(e^{-\frac{d_k(a^*_k, y)^2}{32t}}+ e^{-\frac{2^kd_k(a^*_k,y)}{2}}\right) \cdot 2^{-k} \nonumber
\\
&+&\sum_{\substack{y\in D^k_\eps\\d_k(y, a_k^*)\ge M}} \frac{c_1}{t}\left(e^{-\frac{d_k(a^*_k, y)^2}{32t}}+ e^{-\frac{2^kd_k(a^*_k,y)}{2}}\right)\cdot 2^{-2k}. \label{compute-3.7-1}
\end{eqnarray}
To handle the first term on the right hand side of \eqref{compute-3.7-1},  for $x\in E^k$,  we let $n_{a^*_k}(x)=2^kd_k(x, a^*_k)$, i.e., the smallest number of edges a path has connecting $x$ and $a^*_k$. 
Thus for the first term on the right hand side of \eqref{compute-3.7-1}, it holds for $M\ge 1$ that
\begin{eqnarray}
\sum_{\substack{y\in 2^{-k}\IZ_+\\d_k(y, a_k^*)\ge M}}\frac{c_1}{t} e^{-\frac{2^k d_k(a^*_k, y)}{2}}\cdot 2^{-k} 
&\stackrel{i=n_{a^*_k}(y)}{\le }&\sum_{i=2^k[M]}^\infty\frac{c_2\cdot 2^{-k}}{t}e^{-i/2}  \le \sum_{i=2^k[M]}^\infty\frac{c_2\cdot 2^{-k}}{t}e^{-i/4}\cdot e^{-i/4} \nonumber
\\
(t\in [8^{-k},T])&\le & c_2 \cdot 4^k e^{-2^k/4} \sum_{i=2^k[M]}^\infty e^{- i/4}. \label{compute-3.7-2}
\end{eqnarray}
On the right hand side above, it is clear that $\sup_{k\ge 1}\{4^ke^{-2^k/4}\}<\infty$. Also since $\int_0^\infty e^{-x/4}dx<\infty$, by integral comparison theorem we know that the infinite sum $\sum_{i=0}^\infty e^{-i/4}$ converges. There, the right hand side of \eqref{compute-3.7-2} can be made arbitrarily small as long as  $M$ is sufficiently large.  In addition, to bound the    other term in the line above \eqref{compute-3.7-1},   noticing that there exists some constant $c(M)>0$  decreasing in  $M\ge 1$  such that
\begin{equation}\label{compute-3.7-4}
\frac{1}{t}e^{-\frac{M^2}{64t}}<c(M), \quad \text{for all }t\in (0, T], \quad \lim_{M\uparrow \infty} c(M)=0,
\end{equation}
we have
\begin{eqnarray}
&&\sum_{\substack{y\in 2^{-k}\IZ_+\\d_k(y, a_k^*)\ge M}}\frac{c_1}{t}e^{-\frac{d_k(a^*_k, y)^2}{32t}} \cdot 2^{-k}   \nonumber
\\
&\le & \sum_{\substack{y\in 2^{-k}\IZ_+\\d_k(y, a_k^*)\ge M}}\frac{c_1\cdot 2^{-k}}{t}e^{-\frac{d_k(a^*_k, y)^2}{64t}} \cdot e^{-\frac{d_k(a^*_k, y)^2}{64t}}   \nonumber
\\
&\stackrel{\eqref{compute-3.7-4}}{\le }&  \sum_{\substack{y\in 2^{-k}\IZ_+\\d_k(y, a_k^*)\ge M}} c_1\cdot c(M)\cdot 2^{-k}\cdot e^{-\frac{d_k(a^*_k, y)^2}{64t}} \nonumber
\\
(i=n_{a^*_k}(y)=2^kd_k(a^*_k, y))&\le & \sum_{i=2^k[M]}^\infty c_1\cdot c(M)\cdot 2^{-k}e^{-\frac{i^24^{-k}}{64t}}\nonumber
\\
&\le & c_1  \cdot c(M)\sum_{j=[M]}^\infty \sum_{i=2^kj}^{2^k(j+1)-1}  2^{-k}e^{-\frac{i^24^{-k}}{64t}} \nonumber
\\
& \le & c_1\cdot c(M) \sum_{j=[M]}^\infty e^{-\frac{j^2}{64T}}. \label{compute-3.7-5}
\end{eqnarray}
In view of \eqref{compute-3.7-4},  the right hand side of \eqref{compute-3.7-5} can be made arbitrarily small with sufficiently large $M$. Combining the discussion above, we come to the conclusion that  for the right hand side of \eqref{compute-3.7-1}, given any $\delta>0$, there exists $c_3>0$ sufficiently large, such that for all $k\ge k_0$ and $M>c_3$, it holds that 
\begin{equation}\label{first-term-3.32}
\sum_{\substack{y\in 2^{-k}\IZ_+\\d_k(y, a_k^*)\ge M\cdot 2^k}} \frac{c_1}{t}\left(e^{-\frac{d_k(a^*_k, y)^2}{32t}}+e^{-\frac{2^kd_k(a^*_k,y)}{2}}\right) \cdot 2^{-k}<\delta.
\end{equation}
To handle  the second term on the right hand side of \eqref{compute-3.7-1},  we denote by $S^k_i$ the boundary of the square centered at the origin with four vertices  $(i2^{-k}, i2^{-k}), (i2^{-k}, -i2^{-k}), (-i2^{-k}, -i2^{-k})$, $(-i2^{-k}, i2^{-k})$. $S_i$ contains at most $8i$ points in $E^k$.  In addition, by elementary geomemtry,
\begin{equation}
\{y\in E^k: \; d_k(y, a^*_k)\ge M\}\subset\bigcup_{i=2^k\left[\frac{M+\eps}{\sqrt{2}}\right]}^\infty  S_i,\quad \text{for  } i\ge [\eps2^k]+1.
\end{equation}
Noticing that for any $y\in S_i$,
\begin{equation}\label{e:3.37}
d_k(a^*_k, y)\ge \left(i-[\eps2^k]\right)\cdot 2^{-k}\ge i2^{-k}-[\eps],
\end{equation}
 by selecting $M\ge 16\eps$, we first have
\begin{eqnarray}
&&\sum_{\substack{y\in D^k_\eps\\d_k(y, a_k^*)\ge M}} \frac{c_1}{t}e^{-\frac{2^kd_k(a^*_k,y)}{2}}\cdot 4^{-k} \nonumber
\\
&\le & \sum_{j=2^k\left[\frac{M+\eps}{\sqrt{2}}\right]}^\infty\sum_{y\in  S_j}\frac{c_1}{t} e^{-\frac{2^k d_k(a^*_k, y)}{2}}\cdot 4^{-k}
 \nonumber
 \\
\eqref{e:3.37} &\le & \sum_{j=2^k\left[\frac{M+\eps}{\sqrt{2}}\right]}^\infty c_1\cdot 8j \cdot \frac{ 4^{-k}}{8^{-k}} \exp\left(-\frac{2^k}{2}\left(j2^{-k}-[\eps]\right)\right) \nonumber
 \\
 &= &c_4\cdot 2^k\sum_{j=2^k\left[\frac{M+\eps}{\sqrt{2}}\right]}^\infty j\exp\left(-\frac{2^k}{4}\left(j2^{-k}-[\eps]\right)\right)   \exp\left(-\frac{2^k}{4}\left(j2^{-k}-[\eps]\right)\right)  \nonumber
 \\
 &\le & c_4\cdot 2^k\exp\left[-\frac{2^k\left(\left[\frac{M+\eps}{\sqrt{2}}\right]-[\eps]\right)}{4}\right] \sum_{j=2^k\left[\frac{M+\eps}{\sqrt{2}}\right]}^\infty j\exp\left[-\frac{2^k}{4}\left(j2^{-k}-[\eps]\right)\right] \nonumber
 \\
 &= & c_4\cdot 2^k\exp\left[-\frac{2^k\left(\left[\frac{M+\eps}{\sqrt{2}}\right]-[\eps]\right)}{4}\right] \sum_{j=2^k\left[\frac{M+\eps}{\sqrt{2}}\right]}^\infty j \exp\left(-\frac{j-2^k[\eps]}{4}\right)\nonumber
 \\ 
(M\ge 16\eps) &\le & c_4\cdot 2^k e^{-\frac{2^kM}{16}} \sum_{j=2^k\left[\frac{M}{2}\right]}^\infty j e^{-\frac{j}{8}}.\label{compute-3.7-3}
\end{eqnarray}
We notice that  $\sup_{k\ge 1}2^k e^{-2^k/16}<\infty$, $\sum_{j=1}^\infty  je^{-j/8}<\infty$, Therefore, since the right hand side of \eqref{compute-3.7-3} decreases to zero as $M\uparrow \infty$, it can be made arbitrarily small with sufficiently large $M$. 
To finish bounding the second term on the right hand side of  \eqref{compute-3.7-1}, we again first note that for any $M>16\eps>0$, there exists  $c(M)\downarrow 0$ as $M\uparrow +\infty$  such that
\begin{equation}\label{compute-3.7-6}
\sup_{x\ge M, 0<t<T} \frac{x+1}{t}e^{-\frac{(x-\eps)^2}{64t}}<c(M).
\end{equation}
Hence we have
\begin{eqnarray}
&&\sum_{\substack{y\in D^k_\eps\\d_k(y, a_k^*)\ge M}} \frac{c_1}{t}e^{-\frac{d_k(a^*_k,y)^2}{32t}}\cdot 4^{-k} \nonumber
\\
&\le & \sum_{j=2^k\left[\frac{M+\eps}{\sqrt{2}}\right]}^\infty\sum_{y\in  S_j}\frac{c_1}{t} e^{-\frac{d_k(a^*_k, y)^2}{32t}}\cdot 4^{-k} \nonumber
\\
&\le & c_1 \sum_{j=2^k\left[\frac{M+\eps}{\sqrt{2}}\right]}^\infty  \sum_{y\in  S_j} \frac{4^{-k}}{t}\exp\left[\frac{\left(j2^{-k}-[\eps]\right)^2}{32t}\right]  \nonumber
\\
&\le & c_1 \sum_{j=2^k\left[\frac{M+\eps}{\sqrt{2}}\right]}^\infty     \frac{4^{-k}\cdot 8j}{t}  \exp\left[\frac{\left(j2^{-k}-[\eps]\right)^2}{32t}\right]\nonumber
\\
&\le & c_1 \sum_{u=\left[\frac{M+\eps}{\sqrt{2}}\right]}^\infty \sum_{j=2^ku}^{2^k(u+1)-1}\frac{4^{-k}\cdot 8j}{t}  \exp\left[\frac{\left(j2^{-k}-[\eps]\right)^2}{32t}\right]\nonumber
\\
&\le & 8c_1\sum_{u=\left[\frac{M+\eps}{\sqrt{2}}\right]}^\infty \frac{4^{-k}\cdot 2^k(u+1)}{t}e^{-\frac{(u-\eps)^2}{64t}}e^{-\frac{(u-\eps)^2}{64t}} \nonumber
\\ 
&\stackrel{\eqref{compute-3.7-6}}{\le } &8c_1\cdot c(M) \sum_{u=\left[\frac{M+\eps}{\sqrt{2}}\right]}^\infty e^{-\frac{(u-\eps)^2}{64T}}.\label{compute-3.7-7}
\end{eqnarray}
Again, the right hand side of \eqref{compute-3.7-7} can be made arbitrarily small with sufficiently large $M$. This combined with the earlier discussion shows that for any $\delta>0$,  there exist $c_5>0$ sufficiently large, such that for all $k\ge k_0$ and $M>c_5$,
\begin{equation}\label{second-term-3.32}
\sum_{\substack{y\in D^k_\eps\\d_k(y, a_k^*)\ge M}} \frac{c_1}{t}\left(e^{-\frac{d_k(a^*_k, y)^2}{64t}}+e^{-\frac{2^kd_k(a^*_k,y)}{4}}\right) \cdot 4^{-k}<\delta.
\end{equation}
Replacing the two terms on the right hand side of \eqref{compute-3.7-1} with the two upper bounds \eqref{first-term-3.32} and \eqref{second-term-3.32} respectively, we have shown that given any $\delta>0$, for all $k\ge k_0$ and all $M\ge \max\{c_3, c_5\}$, 
\begin{align*}
\IP^{a^*_k}\left[ d_k(X_t^k, a^*_k)\ge M\right]  \le 2\delta. 
\end{align*}
This completes the proof.
\end{proof}
The next proposition justifies the first tightness condition  in Proposition \ref{tightness-criterion} for $\{X^k\}_{k \ge 1}$. 
\begin{prop}\label{P:3.8}
For any fixed  $T>0$, $\delta>0$, there exist $k_1\in \mathbb{N}$ and $M_3>0$ such that for all $k\ge k_1$,
\begin{equation*}
\IP^{a^*_k}\left[\sup_{t\in [0, T]}\big|X^k_t\big|_\rho>M\right]<\delta. 
\end{equation*} 
\end{prop}

\begin{proof}
We first note  Proposition \ref{P:3.7} implies that:  Given $T>0$ fixed, for every $\delta>0$, there is $M_2>0$ such that
\begin{equation*}
\IP^{a^*_k}\left[d_k(X_T^k, a^*_k)\ge M_2\right]<\delta.
\end{equation*}
Since $\rho(x, a^*)\le d_k(x, a^*_k)$ for all $x\in E^k\subset E$, this further implies that
\begin{equation}\label{compute-P3.8-12}
\IP^{a^*_k}\left[\left|X^k_T\right|_\rho\ge M_2\right]<\delta.
\end{equation}
In view of Lemma \ref{L:strong-markov}, it suffices that in the following we show that for every $T>0$ and every $\delta>0$, there exists an $M>0$ and $n_1\in \mathbb{N}$ such that for all $k\ge n_1$,
\begin{description}
\item{(i)} $\IP^{a^*_k}\left[ \sup_{t\in [0, 8^{-k}]} |X^k_t|_\rho\ge M \right]<\delta$, and $\IP^{a^*_k}\left[ T-8^{-k}\le \tau_M  \le T,\, \left|X^k_T\right|_\rho\le \frac{M}{2}\right]<\delta $;
\item{(ii)} $\IP^{a^*_k}\left[8^{-k}\le \tau_M\le T-8^{-k},\, \left|X^k_T\right|_\rho\le  \frac{M}{2} \right]<\delta$,
\end{description}
where $\tau_M=\inf\{t>0: \;|X^k_t|_\rho\ge M\}$.
The first statement of (i) is proved in Proposition \ref{P:3.6}.  In order to claim the second statement of (i), for the $M_1$  specified in Proposition \ref{P:3.6}, when $M>2M_1$ and $k\ge k_0$, on account of strong Markov property, 
\begin{eqnarray}
&&\IP^{a^*_k}\left[ T-8^{-k}\le \tau_M  \le T,\, \left|X^k_T\right|_\rho\le \frac{M}{2}\right] \nonumber
\\
&\le &\IE^{a^*_k}\left[ \IP^{X^k_{\tau_M}}\left[ \sup_{t\in [0, 8^{-k}]} |X^k_t|_\rho\le \frac{M}{2}  \right],  T-8^{-k}\le \tau_M \le M\right]\nonumber
\\
&\le & \sup_{|y|_\rho\ge M} \IP^y\left[\sup_{t\in [0, 8^{-k}]}   |X^k_t|_\rho\le\frac{M}{2}\right]\nonumber
\\
&\le & \sup_{|y|_\rho\ge M} \IP^y\left[\sup_{t\in [0, 8^{-k}]}   \rho\left(X^k_t, X^k_0\right)\geq\frac{M}{2}\right] \nonumber
\\
&\le & \sup_{|y|_\rho\ge M} \IP^y\left[\sup_{t\in [0, 8^{-k}]}   \rho\left(X^k_t, X^k_0\right)\geq M_1 \right]<\delta.
\label{computation-3.8-1}
\end{eqnarray}
This establishes (i). To verify (ii), again by strong Markov property, we have for any $M>0$ that 
\begin{eqnarray}
&&\IP^{a^*_k}\left[8^{-k}\le \tau_M\le T-8^{-k},\, \left|X^k_T\right|_\rho\le  \frac{M}{2} \right]  \nonumber
\\
&=&\int_{8^{-k}}^{T-8^{-k}}\IE^{a^*_k}\left[ \IP^{X^k_s}\left[\left|X^k_{T-s}\right|\le \frac{M}{2}\right]; \tau_M\in ds\right] \nonumber
\\
&\le &  \int_{8^{-k}}^{T-8^{-k}} \IE^{a^*_k}\left[ \sup_{t\in [8^{-k}, T-8^{-k}]} \IP^{X^k_s}\left[|X^k_t|_\rho\le \frac{M}{2} \right];\tau_M\in ds \right] \nonumber
\\
&\le & \sup_{\substack{|y|_\rho\ge M\\ t\in [8^{-k}, T-8^{-k}]}}\IP^y\left[|X^k_t|_\rho\le \frac{M}{2}\right].
 \label{compute-3.8-2}
\end{eqnarray}
To bound the right hand side of \eqref{compute-3.8-2}, we utilize Proposition \ref{HKUB}. Given any $k\ge k_0$, for any $y\in E^k$ such that $|y|_\rho \ge M$ and any $t\in [8^{-k}, T-8^{-k}]$,
\begin{equation}\label{compute-3.8-3}
\IP^y\left[|X^k_t|_\rho\le \frac{M}{2}\right]=\sum_{\substack{x\in 2^{-k}\IZ\\ |x|_\rho\le \frac{M}{2}}}p_k(t,y, x)m_k(dx)+\sum_{\substack{x\in D^k_\eps\\ |x|_\rho\le \frac{M}{2}}}p_k(t,y, x)m_k(dx) +p_k(t, y, a^*_k)m_k(a^*_k),
\end{equation}
In the following we give upper bounds to each of the three terms on the right hand side of \eqref{compute-3.8-3} respectively. We note that for $|y|_\rho\ge M$, $d_k(y, x)\ge \rho(y, x)\ge M/2$ when $|x|_\rho\le M/2$.   Also note that $\#\{x:x\in 2^{-k}\IZ_+,|x|_\rho \le M/2\}\le M2^k$. Therefore, or the first term on the right hand side of \eqref{compute-3.8-3}, by  Corollary \ref{HKUB}  we have 
\begin{eqnarray}
 \sum_{\substack{x\in 2^{-k}\IZ_+\\ |x|_\rho\le \frac{M}{2}}}p_k(t,y, x)m_k(dx) 
&\le & \sum_{\substack{x\in 2^{-k}\IZ\\ |x|_\rho\le \frac{M}{2}}}  \frac{c_1}{t}\left(e^{-\frac{d_k(y, x)^2}{32t}}+e^{-\frac{2^kd_k(y, x)}{2}}\right) \cdot 2^{-k} \nonumber
\\
(d_k(y, x)\ge M/2) &\le &\sum_{\substack{x\in 2^{-k}\IZ_+ \\ |x|_\rho\le \frac{M}{2}}}  \frac{c_12^{-k}}{t}\left(  e^{-\frac{M^2}{128t}}+e^{-\frac{2^kM}{4}}\right) \nonumber
\\
&\le & 2^kM\cdot \frac{c_12^{-k}}{t}\left(  e^{-\frac{M^2}{128t}}+e^{-\frac{2^kM}{4}}\right) \nonumber
\\
&=&  \frac{c_1M}{t}e^{-\frac{M^2}{128t}}+\frac{c_1M}{t}e^{-\frac{2^kM}{4}}.
 \label{compute-3.8-4} 
\end{eqnarray}
To estimate the two terms on the right hand side of \eqref{compute-3.8-4}, we note that given $M>0$ and $T>0$ fixed,   there exists some $c(M)<\infty$ decreasing in $M$ such that
\begin{equation}\label{compute-3.8-5}
\sup_{x\ge M, 0<t<T}\frac{1}{t}e^{-x^2/(256t)}<c(M).
\end{equation}
Thus for $t\in [8^{-k}, T-8^{-k}]$,
\begin{align}
\frac{c_1M}{t}e^{-\frac{M^2}{128t}} &= \frac{c_1M}{t}e^{-\frac{M^2}{256t}} \cdot e^{-\frac{M^2}{256t}} \stackrel{\eqref{compute-3.8-5}}{\le }   c_1\cdot M\cdot c(M) e^{-M^2/(256T)},
\end{align} 
which can be made arbitrarily small by choosing $M$ sufficiently large. For the second term on the right hand side of \eqref{compute-3.8-4}, noticing that $t\in [8^{-k}, T-8^{-k}]$,   we have
\begin{align}
\frac{c_1M}{t}e^{-\frac{2^kM}{4}} \le c_1M8^ke^{-\frac{2^kM}{4}}=\frac{(2^kM)^3}{M^2}e^{-\frac{2^kM}{4}},\label{compute-3.8-6}
\end{align} 
which again, regardless of the value of $k\ge k_0$, can be made arbitrarily small as long as  $M$ is chosen sufficiently large, regardless of the value of $k\ge k_0$, in view of the fact that $\lim_{x\rightarrow +\infty}x^3e^{-x/4}=0$. The discussion above shows that both terms on the right hand side of \eqref{compute-3.8-4} can be made arbitrarily small with sufficiently large $M$, i.e., given any $\delta>0$, there exists $c_2>0$ such that for all $k\ge k_0$ and all $M>c_2$, 
\begin{equation}\label{compute-3.8-8}
\sum_{\substack{x\in 2^{-k}\IZ\\ |x|_\rho\le \frac{M}{2}}}p_k(t,y,x)m_k(dx) <\delta. 
\end{equation}
Now we take care of the second term on the right hand side of \eqref{compute-3.8-3}. We again denote by $S_i$ the boundary of the square centered at the origin with four vertices  $(i2^{-k}, i2^{-k})$, $(i2^{-k}, -i2^{-k})$,  $(-i2^{-k}, -i2^{-k})$, $(-i2^{-k}, i2^{-k})$.  Similarly, since $|y|_\rho>M$ and $|x|_\rho\le M/2$, provided that $M\ge 4\eps$, it holds
\begin{eqnarray}
\sum_{\substack{x\in D^k_\eps\\ |x|_\rho\le \frac{M}{2}}}p_k(t,y,x)m_k(dx) 
&\le & \sum_{\substack{x\in D^k_\eps\\ |x|_\rho\le \frac{M}{2}}}  \frac{c_1}{t}\left(e^{-\frac{d_k(y,x)^2}{32t}}+e^{-\frac{2^kd_k(y,x)}{2}}\right) \cdot 4^{-k} \nonumber
\\
&\le & \sum_{\substack{x\in D^k_\eps\\ |x|_\rho\le \frac{M}{2}}}  \frac{c_1}{t}\left(e^{-\frac{M^2}{128t}}+e^{-\frac{2^kM}{4}}\right) \cdot 4^{-k} \nonumber
\\
&\le & \sum_{j=1}^{2^k\left(\left[M/2+\eps\right]+1\right)}\sum_{x\in S_j} \frac{c_14^{-k}}{t}\left(e^{-\frac{M^2}{128t}}+e^{-\frac{2^kM}{4}}\right) \nonumber
\\
(\#\{x: x\in S_j\}\le 8j)&\le & \sum_{j=1}^{2^k\left(\left[M/2+\eps\right]+1\right)} \frac{c_14^{-k}8j}{t}\left(e^{-\frac{M^2}{128t}}+e^{-\frac{2^kM}{4}}\right) \nonumber
\\
&\le & \sum_{u=0}^{\left(\left[M/2+\eps\right]+1\right)}\sum_{j=2^ku}^{2^k(u+1)-1}\frac{c_14^{-k}8j}{t}\left(e^{-\frac{M^2}{128t}}+e^{-\frac{2^kM}{4}}\right) \nonumber
\\
&\le & \sum_{u=0}^{\left(\left[M/2+\eps\right]+1\right)}\frac{c_12^{-k}8(u+1)}{t}\left(e^{-\frac{M^2}{128t}}+e^{-\frac{2^kM}{4}}\right)\nonumber
\\
&\le &  \frac{c_32^{-k}M^2}{t}\left(e^{-\frac{M^2}{128t}}+e^{-\frac{2^kM}{4}}\right).\label{compute-3.8-7}
\end{eqnarray}
To bound the two terms on the right hand of \eqref{compute-3.8-7} respectively, we first notice that  there exists some constant $c_4>0$ such that  $\sup_{x\ge 0}xe^{-x/128}\le c_4$. Thus
\begin{align}
\frac{c_32^{-k}M^2}{t} e^{-\frac{M^2}{128t}}\le c_3c_42^{-k},
\end{align}
which can be made arbitrarily small with sufficiently large $k$. For the other term on the right hand side of \eqref{compute-3.8-7}, noticing that $t\ge 8^{-k}$,
\begin{align*}
\frac{c_32^{-k}M^2}{t} e^{-\frac{2^kM}{4}}\le c_34^kM^2e^{-\frac{2^kM}{4}}\le c_3\left(2^kM\right)^2e^{-\frac{2^kM}{4}},
\end{align*}
which again can be made arbitrarily small as long as $M$ is sufficiently large, regardless of the value of $k\ge k_0$, because $\displaystyle{\lim_{x\rightarrow +\infty}x^2e^{-x/4}\rightarrow 0}$. Combining the disccusion above regarding both terms on the right hand side of \eqref{compute-3.8-7}, we see that given any $\delta>0$, there exists an integer $n_1\ge k_0$ and $c_5>0$, such that for all $k\ge n_1$ and all $M>c_5$, 
\begin{equation}\label{compute-3.8-9}
\sum_{\substack{x\in D^k_\eps\\ |x|_\rho\le \frac{M}{2}}}p_k(t,x,y)m_k(dx) <\delta. 
\end{equation}
Finally, for the third term on the right hand side of \eqref{compute-3.8-3}, since $|y|_\rho\ge M$, 
\begin{eqnarray}
p_k(t, y, a^*_k)m_k(a^*_k) &\stackrel{\eqref{mass-ak-star}}{\le} & \frac{c_1}{t}\left(e^{-\frac{d_k(y,a^*_k)^2}{32t}}+e^{-\frac{2^kd_k(y,a^*_k)}{2}}\right) \cdot 2^{-k} \le   \frac{c_12^{-k}}{t}\left(  e^{-\frac{M^2}{128t}}+e^{-\frac{2^kM}{4}}\right).\nonumber
\end{eqnarray}
From here, using the same argument as that for \eqref{compute-3.8-8}, it can be shown that given any $\delta>0$, there exists there exists $c_6>0$ such that for all $k\ge k_0$ and all $M>c_6$, 
\begin{equation}\label{P3.8-compute-10}
p_k(t,y,a^*_k)m_k(a^*_k) <\delta. 
\end{equation}
Combining \eqref{compute-3.8-8}, \eqref{compute-3.8-9}, and \eqref{P3.8-compute-10}, in view of \eqref{compute-3.8-3}, we have showed that given any $\delta>0$, for all $k\ge n_1$ and all $M>\max\{c_2, c_5, c_6\}$, it holds that 
\begin{equation}
\sup_{\substack{|y|_\rho\ge M\\t\in [8^{-k}, T-8^{-k}]}}\IP^y\left[|X^k_t|_\rho\le \frac{M}{2}\right]<3\delta.
\end{equation}
In view of \eqref{compute-3.8-2}, we have verified the condition (ii) stated at the beginning of this proof. Now that both conditions (i) and (ii) have been verified, the proof is complete. 
\end{proof} 

Before we establish the second tightness condition in Proposition \ref{tightness-criterion} for $\{X^k\}_{k \ge 1}$, we need the following lemma. 

\begin{lem}\label{P:3.9}
For any $T>0$, $\delta_1, \delta_2>0$ given, there exist  $\delta_3>0$ and  $k_2\in \mathbb{N}$ such that  for all $k\ge k_2$,
\begin{equation*}
\sup_{x\in E}\sup_{t\in [8^{-k}, \delta_3]} \left(\frac{1}{\delta_3}+1\right) \IP^x\left[\rho\left(X_0^k, X^k_t\right)\ge \delta_1 \right]<\delta_2.
\end{equation*}
\end{lem}

\begin{proof}
The idea of this proof is similar to that of Proposition \ref{P:3.7}. By the definition of $d_k$, for any $x\in E$ and  any $t\in [8^{-k}, \delta_3]$,
\begin{eqnarray}
&&\left(\frac{1}{\delta_3}+1\right)\IP^x\left[\rho\left(X_0^k, X^k_t\right)\ge \delta_1 \right] \nonumber
\\
&\le & \left(\frac{1}{\delta_3}+1\right)\IP^x\left[d_k\left(X_0^k, X^k_t\right)\ge \delta_1 \right] \nonumber
\\
&\le & \left(\frac{1}{\delta_3}+1\right)\sum_{\substack{y\in 2^{-k}\IZ_+\\d_k(x,y)\ge \delta_1 }}  \frac{c_1}{t}\left(e^{-\frac{d_k(x, y)^2}{32t}}+e^{-\frac{2^kd_k(x,y)}{2}}\right) \cdot 2^{-k} \nonumber
\\
&+& \left(\frac{1}{\delta_3}+1\right)\sum_{\substack{y\in D^k_\eps\\d_k(x,y)\ge \delta_1 }}  \frac{c_1}{t}\left(e^{-\frac{d_k(x, y)^2}{32t}}+e^{-\frac{2^kd_k(x,y)}{2}}\right) \cdot 4^{-k} \nonumber
\\
&+& \mathbf{1}_{\{d_k(x, a^*_k)\ge \delta_1\}}(x)\cdot\left(\frac{1}{\delta_3}+1\right)\frac{c_1}{t}\left(e^{-\frac{d_k(x, a^*_k)^2}{32t}}+e^{-\frac{2^kd_k(x,a^*_k)}{2}}\right)m_k(a^*_k)\nonumber
\\
&=&(I)+(II)+(III).\label{compute-3.9-1}
\end{eqnarray}

Now  we need to claim that given any $\delta_1, \delta_2>0$, there exist $\delta_3>0$ and $n_1\in \mathbb{N}$ such that for all $k\ge n_1$, all three terms (I)-(III) on  the right hand side of  \eqref{compute-3.9-1} are smaller than $\delta_2$. Towards this purpose, for the first term in (I), we let $n_k(x,y):=d_k(x,y)2^k$, i.e., the smallest number of edges between $x$ and $y$ in $E^k$. For $k$ sufficiently large such that $\delta_12^{k/4}\ge 1$, noticing that there exists some $c_2>0$ such that
\begin{equation}\label{compute-3.9-16}
\sup_{k\ge 1}4^k\cdot \exp\left(-2^{\frac{3k}{16}}\right)\le c_2,
\end{equation}
we have
\begin{eqnarray}
\sum_{\substack{y\in 2^{-k}\IZ_+\\d_k(x,y)\ge \delta_1 }} \frac{c_12^{-k}}{t}e^{-\frac{2^kd_k(x, y)}{2}} &= & \sum_{\substack{y\in 2^{-k}\IZ_+\\n_k(x,y)\ge \delta_12^k}}\frac{c_12^{-k}}{t}e^{-\frac{n_k(x,y)}{2}} \nonumber
\\
(i=n_k(x,y))&=& \sum_{i=[\delta_12^k]+1}^\infty \frac{c_12^{-k}}{t}e^{-\frac{i}{2}} \nonumber
\\
(t\ge 8^{-k})& \le & \sum_{i=[\delta_12^k]}^\infty c_14^k e^{-\frac{i}{4}} e^{-\frac{i}{4}} \nonumber
\\
(\delta_12^{k/4}>1)&\le & \sum_{i=[2^{(3k)/4}]}^\infty c_1 4^{k} e^{-\frac{i}{4}} e^{-\frac{i}{4}} \nonumber
\\
&\le & \sup_{k\ge 1} \left(4^k \exp\left(-2^{\frac{3k}{16}}\right)\right) \sum_{i=[2^{(3k)/4}]}^\infty c_1 e^{-\frac{i}{4}} \nonumber
\\
\eqref{compute-3.9-16}&\le & c_2 \sum_{i=[2^{(3k)/4}]}^\infty c_1 e^{-\frac{i}{4}}.\label{compute-3.9-9}
\end{eqnarray}
 For the second term in (I) on the right hand side of \eqref{compute-3.9-1}, since there exists some $c_3>0$ only depending on  $\delta_1$ such that  
 \begin{equation}\label{compute-3.9-17}
 \sup_{t>0}\frac{1}{t}e^{-\frac{\delta_1^2}{64t}}\le c_3,
 \end{equation}
  we have
 \begin{align}
\sum_{\substack{y\in 2^{-k}\IZ_+\\d_k(x,y)\ge \delta_1 }} \frac{c_1}2^{-k}{t}e^{-\frac{d_k(x,y)^2}{32t}}  &\le  \sum_{\substack{y\in 2^{-k}\IZ_+\\d_k(x,y)\ge \delta_1 }}\frac{c_1 2^{-k}}{t}e^{-\frac{d_k(x,y)^2}{64t}} e^{-\frac{d_k(x,y)^2}{64t}} \nonumber
 \\
 &\le  \left(\sup_{8^{-k}\le t\le \delta_3}\frac{c_1}{t}e^{-\frac{\delta_1^2}{64t}}\right)2^{-k}\sum_{\substack{y\in 2^{-k}\IZ_+\\n_k(x,y)\ge \delta_12^{k} }} c_1 e^{-\frac{n_k(x,y)^24^{-k}}{64t}} \nonumber
 \\
 (\eqref{compute-3.9-17}, i=n_k(x,y))&\le  c_3 \cdot 2^{-k} \sum_{i=[\delta_12^{k}]}^\infty c_1 e^{-\frac{i^24^{-k}}{64t}}\nonumber
 \\
 &\le c_3\cdot 2^{-k} \sum_{u=[\delta_1]}^\infty \sum_{i=u2^k}^{(u+1)2^k-1} c_1 e^{-\frac{i^24^{-k}}{64t}}\nonumber
 \\
 &\le c_3\cdot 2^{-k} \sum_{u=[\delta_1]}^\infty \left(2^k\cdot c_1  e^{-\frac{u^2}{64t}}\right)\nonumber
 \\
( t\le \delta_3)   &\le  c_3 \sum_{u=[\delta_1]}^\infty   e^{-\frac{u^2}{64\delta_3}}, \label{compute-3.9-10}
 \end{align}
Combining \eqref{compute-3.9-9} and \eqref{compute-3.9-10}, we have shown that for any given $\delta_1>0$, for all $k$ large enough such that $(2^{1/4})^k\delta_1>1$, there exists $c_4>0$ such that
\begin{eqnarray}
&& \left(\frac{1}{\delta_3}+1\right)\sum_{\substack{y\in 2^{-k}\IZ_+\\d_k(x,y)\ge \delta_1 }}  \frac{c_1}{t}\left(e^{-\frac{d_k(x, y)^2}{32t}}+e^{-\frac{2^kd_k(x,y)}{2}}\right) \cdot 2^{-k}  \nonumber
\\
&\le & c_4\left(\frac{1}{\delta_3}+1\right)\left(\sum_{i=[2^{(3k)/4}]}^\infty e^{-\frac{i}{4}}+\sum_{u=[\delta_1]}^\infty   e^{-\frac{u^2}{64\delta_3}}\right).\label{compute-3.9-11}
\end{eqnarray}
Now we take care of (II) on the right hand side of \eqref{compute-3.9-1}. First of all, we claim that for any given $x\in E^k$   any $j\in \mathbb{N}$, 
\begin{equation}\label{compute-3.9-15}
\#\left\{ y\in D^k_\eps: \, d_k(x,y) \le j\cdot 2^{-k} \right\}\le 256\left(j^2+\eps^24^k\right).
\end{equation}
To see \eqref{compute-3.9-15}, it suffices to note that for any  $y_1,y_2\in \left\{ y\in D^k_\eps, \, d_k(x,y) \le j2^{-k} \right\}$, if we denote by $n_k(y_1, y_2)$ the smallest number of edges between $y_1$ and $y_2$, then it must hold that $n_k(y_1, y_2)\le 2j$. Therefore $|y_1 - y_2|\le 2j2^{-k}+2\eps$. Thus the Lebesgue measure of the set $\left\{ y\in D^k_\eps: \, d_k(x,y) \le j2^{-k} \right\}$ is at most $\pi \left(2j2^{-k}+2\eps\right)^2$. Since any two points in this set is at least $2^{-k}$ Euclidean distance apart, any two discs with Euclidean radius $2^{-k-2}$ centered at two distinct  points in the set $\left\{ y\in D^k_\eps: \, d_k(x,y) \le j \right\}$ must be disjoint. Thus
\begin{align*}
\#\left\{ y\in D^k_\eps: \, d_k(x,y) \le j2^{-k} \right\}\le \frac{\pi (2j2^{-k}+2\eps)^2}{\pi 2^{-2k-4}} \le128\left(j^2+j\eps2^k+\eps^24^k\right)\le256\left(j^2+\eps^24^k\right).
\end{align*}
This verifies \eqref{compute-3.9-15}. Now for the first term in (II) on the right hand side of \eqref{compute-3.9-1}, it holds that
\begin{eqnarray}
&&\sum_{\substack{y\in D^k_\eps\\d_k(x,y)\ge \delta_1 }}  \frac{c_1}{t}e^{-\frac{d_k(x, y)^2}{64t}}\cdot 4^{-k}  \nonumber
\\
&\le & \sum_{j=0}^\infty \sum_{\substack{y\in D^k_\eps\\ \delta_1+j\le d_k(x,y)<\delta_1+j+1}}\frac{c_14^{-k}}{t}e^{-\frac{d_k(x,y)^2}{64t}} \nonumber
\\
\eqref{compute-3.9-15}&\le & \sum_{j=0}^\infty \frac{c_14^{-k}}{t} \left[\left((\delta_1+j+1)2^k\right)^2+4^k\right]e^{-\frac{(\delta_1+j)^2}{64t}} \nonumber
\\
&\le & \sum_{j=0}^\infty \frac{c_14^{-k}}{t}\left((\delta_1+j+1)^24^k+4^k\right)e^{-\frac{(\delta_1+j)^2}{64t}} \nonumber
\\
&\le & \sum_{j=0}^\infty \frac{2c_1}{t}\left(\delta_1+j+1\right)^2e^{-\frac{(\delta_1+j)^2}{64t}} \nonumber
\\
&\le &  4c_1\left(\delta_1+1\right)^2 \sum_{j=0}^\infty \frac{1}{t}e^{-\frac{(\delta_1+j)^2}{64t}} + 4c_1\sum_{j=0}^\infty \frac{j^2}{t}e^{-\frac{(\delta_1+j)^2}{64t}} \nonumber
\\
&\le & 4c_1(\delta_1+1)^2\frac{1}{t}e^{-\frac{\delta_1^2}{64t}}+8c_1(\delta_1+1)^2\sum_{j=0}^\infty
\frac{j^2}{t}e^{-\frac{(\delta_1+j)^2}{64t}} 
.\label{compute-3.9-7}
 \end{eqnarray}
 for any $\delta_1, \delta_2>0$ given, we first note that 
\begin{eqnarray}
\left(\sup_{0<t<\delta_3}\sum_{j=0}^\infty \frac{j^2}{t}e^{-\frac{(\delta_1+j)^2}{64t}}\right) &\le &   \left( \sup_{0<t<\delta_3}    \frac{j^2}{t}  e^{-\frac{\delta_1^2}{64t}}\sum_{j=0}^\infty j^2e^{-\frac{j^2}{64t}} \right) \nonumber
\\
&\le & \sum_{j=0}^\infty j^2e^{-\frac{j^2}{64\delta_3}} \left( \sup_{t>0}    \frac{1}{t}  e^{-\frac{\delta_1^2}{64t}}\right) \nonumber
\\
\eqref{compute-3.9-17} &\le & c_3 \sum_{j=0}^\infty j^2e^{-\frac{j^2}{64\delta_3}}.\label{compute-3.9-31}
\end{eqnarray}
Replacing the last summation term on the right hand side of \eqref{compute-3.9-7} with \eqref{compute-3.9-31}, we get
\begin{eqnarray}
&&\sum_{\substack{y\in D^k_\eps\\d_k(x,y)\ge \delta_1 }}  \frac{c_1}{t}e^{-\frac{d_k(x, y)^2}{64t}}\cdot 4^{-k}  \le 4c_1(\delta_1+1)^2\frac{1}{t}e^{-\frac{\delta_1^2}{64t}}+8c_5(\delta_1+1)^2\sum_{j=0}^\infty j^2e^{-\frac{j^2}{64\delta_3}}.
\end{eqnarray}
For the second term in (II) on the right hand side of \eqref{compute-3.9-1}, we have
\begin{eqnarray}
&&\sum_{\substack{y\in D^k_\eps\\d_k(x,y)\ge \delta_1 }}  \frac{c_1}{t}e^{-\frac{2^kd_k(x, y)}{4}}\cdot 4^{-k} \nonumber
\\
&\le & \sum_{j=0}^\infty \sum_{\substack{y\in D^k_\eps\\ \delta_1+j\le d_k(x,y)<\delta_1+j+1}}\frac{c_14^{-k}}{t}e^{-\frac{2^kd_k(x,y)}{4}} \nonumber
\\
\eqref{compute-3.9-15}&\le & \sum_{j=0}^\infty \frac{c_14^{-k}}{t}\left[\left((\delta_1+j+1)2^k\right)^2+4^k\right]e^{-\frac{\delta_1\cdot 2^{k}+j\cdot 2^k}{4}}  \nonumber
\\
(t\ge 8^{-k})&\le & \sum_{j=0}^\infty 4 c_12^k \left(\delta_1^24^k +j^2\cdot 4^k+4^k\right) e^{-\frac{\delta_1\cdot 2^{k}+j\cdot 2^k}{4}} \nonumber
\\
&\le &\sum_{j=0}^\infty 4c_1\left[(\delta_1^2+1)8^k+j^2\cdot 8^k\right]e^{-\frac{(\delta_1+j)2^k}{4}}. \label{compute-3.9-5}
\end{eqnarray}
Finally, for (III) on the right hand side of \eqref{compute-3.9-1}, on account of Proposition \ref{P:2.3}, we have for $k\ge k_0$  and $t\in [8^{-k}, \delta_3]$  that
\begin{eqnarray}
&& \mathbf{1}_{\{d_k(x, a^*_k)\ge \delta_1\}}(x)\cdot\left(\frac{1}{\delta_3}+1\right)\frac{c_1}{t}\left(e^{-\frac{d_k(x, a^*_k)^2}{32t}}+e^{-\frac{2^kd_k(x,a^*_k)}{2}}\right)m_k(a^*_k)\nonumber
\\
(\text{Proposition} \ref{P:2.3})&\le & \left(\frac{1}{\delta_3}+1\right)\frac{c_1\cdot 2^{-k}}{t}\left(e^{-\frac{\delta_1^2}{32t}}+e^{-\frac{2^k\delta_1}{2}}\right)\nonumber
\\
(t\ge 8^{-k})&\le & \left(\frac{1}{\delta_3}+1\right)\left(\frac{c_1}{t}e^{-\frac{\delta_1^2}{32t}}+c_14^k e^{-\frac{2^k \delta_1}{2}}\right).
 \label{compute-3.9-20}
\end{eqnarray}
Now combining the discussion for (I)-(III) above, i.e., replacing the right hand side of \eqref{compute-3.9-1} with the right hand side terms of  \eqref{compute-3.9-11}, \eqref{compute-3.9-7},  \eqref{compute-3.9-5}, and \eqref{compute-3.9-20},  we have
\begin{eqnarray}
&&\left(\frac{1}{\delta_3}+1\right)\IP^x\left[\rho\left(X_0^k, X^k_t\right)\ge \delta_1 \right] \nonumber
\\
&\le & c_4\left(\frac{1}{\delta_3}+1\right)\left(\sum_{i=[2^{(3k)/4}]}^\infty e^{-\frac{i}{4}}+\sum_{u=[\delta_1]}^\infty   e^{-\frac{u^2}{64\delta_3}}\right) \nonumber
\\
&+ &\left(\frac{1}{\delta_3}+1\right)\bigg\{4c_1(\delta_1+1)^2\frac{1}{t}e^{-\frac{\delta_1^2}{64t}}+8c_5(\delta_1+1)^2\sum_{j=0}^\infty j^2e^{-\frac{j^2}{64\delta_3}} \nonumber
\\
&+&4c_1  \sum_{j=0}^\infty \left(\delta_1^2+1)8^k+j^2\cdot 8^k\right)\cdot e^{-\frac{(\delta_1+j)2^k}{4}}\bigg\} \nonumber
\\
&+& \left(\frac{1}{\delta_3}+1\right)\left(\frac{c_1}{t}e^{-\frac{\delta_1^2}{32t}}+c_14^k e^{-\frac{2^k \delta_1}{2}}\right).  \label{compute-3.9-30}
\end{eqnarray}
To bound the right hand side of \eqref{compute-3.9-30}, notice that for any pair  of  $\delta_1, \delta_2>0$ given,  we may first select  $\delta_3>0$ sufficiently small so that
\begin{equation}\label{compute-3.9-31}
\left(\frac{1}{\delta_3}+1\right)\left(c_4 \sum_{u=[\delta_1]}^\infty   e^{-\frac{u^2}{64\delta_3}}\right)<\delta_2, \quad \left(\frac{1}{\delta_3}+1\right)\sup_{0<t<\delta_3}\frac{c_1}{t}e^{-\frac{\delta_1^2}{32t}}<\delta_2,
\end{equation}
as well as 
\begin{equation}\label{compute-3.9-32}
\left(\frac{1}{\delta_3}+1\right)\left\{4c_1(\delta_1+1)^2\sup_{0<t<\delta_3}\bigg(\frac{1}{t}e^{-\frac{\delta_1^2}{64t}}\bigg)+\left(\frac{1}{\delta_3}+1\right)8c_5(\delta_1+1)^2\sum_{j=0}^\infty j^2e^{-\frac{j^2}{64\delta_3}}\right\}<\delta_2.
\end{equation}
Then with this $\delta_3>0$  fixed, we then choose $n_1\in \mathbb{N}$ big enough so that for all $k\ge n_1$, 
\begin{equation}\label{compute-3.9-33}
 c_4\left(\frac{1}{\delta_3}+1\right)\left(\sum_{i=[2^{(3k)/4}]}^\infty e^{-\frac{i}{4}}\right)<\delta_2,\quad \left(\frac{1}{\delta_3}+1\right)\left(c_14^k e^{-\frac{2^k \delta_1}{2}}\right)<\delta_2,
\end{equation}
and 
\begin{equation}\label{compute-3.9-34}
\left(\frac{1}{\delta_3}+1\right) 4c_1  \sum_{j=0}^\infty \left(\left((\delta_1^2+1)8^k+j^2\cdot 8^k\right) e^{-\frac{(\delta_1+j)2^k}{4}}\right)<\delta_2.
\end{equation}
Combining \eqref{compute-3.9-30}-\eqref{compute-3.9-34}, it has been shown that for any pair $\delta_1, \delta_2>0$ given, there exists $\delta_3>0$ and $n_1\in \mathbb{N}$ such that for all $k\ge n_1$, 
\begin{equation*}
\left(\frac{1}{\delta_3}+1\right)\IP^x\left[\rho\left(X_0^k, X^k_t\right)\ge \delta_1 \right] <6\delta_2, \quad \text{for all }x\in E, t\in [8^{-k}, \delta_3],
\end{equation*}
which completes the proof. 
\end{proof}

The next proposition justifies the second tightness condition in Proposition \ref{tightness-criterion} for $\{X^k\}_{k\ge 1}$.

\begin{prop}\label{P:3.10}
For any $T>0$, $\delta_1, \delta_2>0$, there exist  $\delta_3>0$ and  $k_3\in \mathbb{N}$ such that for all $k\ge k_3$,
\begin{equation}
\IP^{a^*_k}\left[w_\rho\left(X^k,\delta_3, T\right)>\delta_1\right]<\delta_2, 
\end{equation}
where
\begin{equation*}
w_\rho(x,\, \delta_3,\, T):=\inf_{\{t_i\}} \max_{i} \sup_{s, t\in [t_i, t_{i-1}]} \rho(x(s), x(t)),
\end{equation*}
where $\{t_i\}$ ranges over all possible partitions of the form $0=t_0<t_1<\cdots <t_{n-1}<T\le t_n$ with $\min_{1\le i\le n} (t_i-t_{i-1})\ge \delta_3$ and $n\ge 1$.
\end{prop}

\begin{proof}
Based on Proposition \ref{P:3.9}, we further first claim that given any $T>0$, $\delta_1, \delta_2>0$, there exist $\delta_3>0$ and $n_1\in\mathbb{N}$ such that for all $k\ge n_1$, 
\begin{equation}\label{compute-3.10-1}
\sup_{x\in E}\sup_{t\in [0, \delta_3]}\IP^x\left[\rho\left(X_0^k, X^k_t\right)\ge \delta_1 \right]<\delta_2.
\end{equation}
In view of Proposition \ref{P:3.9}, it suffices to show that given any $T>0$, $\delta_1, \delta_2>0$, there exists  $n_1\in\mathbb{N}$ such that for all $k\ge n_1$, 
\begin{equation}\label{compute-3.10-2}
\sup_{x\in E}\sup_{t\in [0, 8^{-k}]}\IP^x\left[\rho\left(X_0^k, X^k_t\right)\ge \delta_1 \right]<\delta_2.
\end{equation}
In fact, for any $x\in E$, any $k\in \mathbb{N}$ and any $t\in [0, 8^{-k}]$,  by the same computation as that in the proof to Proposition \ref{P:3.6} using Stirling's formula,
\begin{eqnarray}
\IP^x\left[\rho\left( X^k_0, X^k_t\right)>\delta_1\right]&\le & \IP^x\left[\sup_{s\in [0, 8^{-k}]}\rho\left( X^k_0, X^k_t\right)>\delta_1  \right] \nonumber
\\
&\le &  \sum_{j=\delta_1\cdot 2^k}^\infty \left(\frac{8^{-k}4^{k}e}{\delta_1\cdot 2^k}  \right)^j \le  \sum_{j=\delta_1\cdot 2^k}^\infty \left(\frac{e}{\delta_14^k}\right)^j, 
\end{eqnarray}
which proves \eqref{compute-3.10-2}. This combined with Proposition \ref{P:3.9} shows \eqref{compute-3.10-1}.
 For any $\delta_1,\delta_3>0$, in view of the definition of $w_\rho$, by strong Markov property we have 
\begin{eqnarray}
&&\IP^{a^*_k}\left[w_\rho\left(X^k,\delta_3, T\right)>\delta_1\right] \nonumber
\\
&\le &  \IP^{a^*_k}\left[ \sup_{1\le i\le \left[T/\delta_3\right]} \sup_{s,t\in [(i-1)\delta_3, i\delta_3\wedge T]} \rho\left(X_s^k, X^k_t\right)>\delta_1\right] \nonumber
\\
&\le & \IP^{a^*_k}\left[ \bigcup_{i=1}^{\left[T/\delta_3\right]} \bigg\{\sup_{s,t\in [(i-1)\delta_3, i\delta_3\wedge T]} \rho\left(X_s^k, X^k_t\right)>\delta_1\bigg\}\right] \nonumber
\\
(\text{strong Markov property})&\le & \left(\left[\frac{T}{\delta_3}\right]+1  \right)\sup_{x\in E}\IP^x \left[\sup_{s,t\in [0, \delta_3]}\rho\left(X^k_s, X^k_t\right)>\delta_1\right]. \label{compute-3.10-3}
\end{eqnarray}
In order to handle the last display in \eqref{compute-3.10-3},	we first denote by $\tau_{\delta_1/2}^k:=\{t>0,\;\rho(X^k_0, X^k_t)\ge \delta_1/2\}$. It then follows by strong Markov property that for any $x\in E$,
\begin{eqnarray}
&&\IP^x\left[ \sup_{s,t\in [0, \delta_3]}\rho\left(X^k_s, X^k_t\right)\ge \delta_1\right] \nonumber
\\
&\le &\IP^x\left[ \sup_{s\in [0, \delta_3]}\rho\left(X^k_0, X^k_s \right)\ge \frac{\delta_1}{2}\right] \nonumber
\\
&\le & \IP^x\left[ \rho\left(X^k_0, X^k_{\delta_3}\right)\ge \frac{\delta_1}{4} \right]+\IP^x\left[ \tau^k_{\delta_1/2}<\delta_3,\;\rho\left(X^k_0, X^k_{\delta_3} \right)\le \frac{\delta_1}{4} \right] \nonumber
\\
&\le & \IP^x\left[ \rho\left(X^k_0, X^k_{\delta_3}\right)\ge \frac{\delta_1}{4} \right] + \int_{0}^{\delta_3}\IE^x\left[ \IP^{X^k_{\tau_{\delta_1/2}}}\left[\rho\left( X^k_0, X^k_{\delta_3-s} \right)\ge \frac{\delta_1}{4} \right],\; \tau_{\delta_1/2}\in ds  \right] \nonumber
\\
&\le & 2\sup_{\substack{y\in E\\ 0\le s\le \delta_3}}\IP^y\left[ \rho\left( X^k_0, X^k_s\right)\ge \frac{\delta_1}{4} \right]. \label{compute-3.10-5}
\end{eqnarray} 
Replacing the last term on the right hand side of \eqref{compute-3.10-3} with \eqref{compute-3.10-5}, we get that for any $\delta_1, \delta_3>0$, and any $k\in\mathbb{N}$,
\begin{equation}\label{compute-10-8}
\IP^{a^*_k}\left[w_\rho\left(X^k,\delta_3, T\right)>\delta_1\right]\le  2\left(\left[\frac{T}{\delta_3}\right]+1  \right)\sup_{\substack{y\in E\\ 0\le s\le \delta_3}}\IP^y\left[ \rho\left( X^k_0, X^k_s\right)\ge \frac{\delta_1}{4} \right].
\end{equation}
Now we are ready to apply Proposition \ref{P:3.9} to finish the proof. Indeed, by Proposition \ref{P:3.9}, for any $T>0$, given any $\delta_1, \delta_2>0$, there exist $\delta_3>0$ and $n_1\in \mathbb{N}$ such that for all $k\ge n_1$,
\begin{equation}\label{compute-3.10-7}
\left(1+\frac{1}{\delta_3}\right)\sup_{\substack{y\in E\\0\le s\le \delta_3}}\IP^y\left[\rho\left(X^k_0, X^k_s \right)\ge \frac{\delta_1}{4} \right]<\frac{\delta_2}{4(T+1)}. 
\end{equation}
Thus \eqref{compute-10-8} yields that
\begin{eqnarray}
\IP^{a^*_k}\left[w_\rho\left(X^k,\delta_3, T\right)>\delta_1\right]&\le & 2\left(\left[\frac{T}{\delta_3}\right]+1  \right)\sup_{\substack{y\in E\\ 0 \le s\le \delta_3}}\IP^y\left[ \rho\left( X^k_0, X^k_s\right)\ge \frac{\delta_1}{4} \right] \nonumber
\\
&\le & \frac{2\left(T+\delta_3\right)}{\delta_3}\sup_{\substack{y\in E\\ 0 \le s\le \delta_3}}\IP^y\left[ \rho\left( X^k_0, X^k_s\right)\ge \frac{\delta_1}{4} \right] \nonumber
\\
\eqref{compute-3.10-7}&\le & \frac{2\left(T+\delta_3\right)}{\delta_3}\cdot \frac{\delta_2}{4(T+1)}\cdot \frac{\delta_3}{\delta_3+1}  <\delta_2. 
\end{eqnarray}
This completes the proof. 
\end{proof}

\begin{thm}
For every     $T>0$, the laws of $\{X^k, \IP^{a^*_k}, k\ge 1\}$ are C-tight in the Skorokhod space $\mathbf{D}([0, T],   E, \rho)$ equipped with the Skorokhod topology. 
\end{thm}
\begin{proof}
This follows immediately from \cite[Chapter VI, Proposition 3.26]{JS}, in view of Proposition \ref{P:3.8} and Proposition \ref{P:3.10}. 
\end{proof}

\section{Weak limit of random walks on spaces with varying dimension} \label{S:5}

We first  establish the uniform convergence of the generators of $X^k$. The method is similar to that in \cite{Lou1}. For notation convenience, we define the following class of functions $\mathcal{G}$:
\begin{align}
\mathcal{G}:&=\{f:\IR^2\cup \IR_+\rightarrow \IR,\,f|_{B_\eps}=\text{const}=f|_{\IR_+}(0),\, f|_{\IR^2}\in C^3(\IR^2), f|_{\IR_+} \in C^3(\IR_+), \nonumber
\\
&f \text{ is supported on a compact subset of }(E\backslash \{a^*\})\cup B_\eps \}.\label{def-class-G}
\end{align}
Every $f\in \mathcal{G}$ can be uniquely  identified as a function mapping $E$ to $\IR$. Thus for $f\in \mathcal{G}$,  we define
\begin{equation}\label{def-wt-Lk}
\mathcal{L}_k f(x):=2^{2k}\sum_{ \substack{y\in E^k, \\y\leftrightarrow x \text{    in }G^k}} \left(f(y)-f(x)\right)J_k (x, dy), \quad \text{ for }x\in E^k.
\end{equation}
 We also set 
\begin{equation*}
S^{k}:=\{x\in D_\eps^k\cap E^k:\; \bar{v}_{k}(x)=4\}\cup 2^{-k}\IZ_+.
\end{equation*}
 It is easy to see that $\{S^k\}_{k\ge 1}$ is an increasing sequence of sets. Also, it is clear that if a vertex $x\in  E^k\backslash S^k$,  then it must be that either $x=a^*_k$ or $x\leftrightarrow a^*_k$ in $G^k$. By a similar argument as that for \cite[Lemma 2.7]{Lou1}, we have the following lemma.

 \begin{lem}\label{L3:12}
 For every fixed $k_0\in \mathbb{N}$ and every $f\in \mathcal{G}$, $\mathcal{L}_kf$ converges uniformly to 
\begin{equation}\label{def-L}
\mathcal{L}f:=\frac{1}{2}\Delta f|_{\IR_+}+\frac{1}{4}\Delta f|_{D_\eps} \quad \text{on } S^{k_0} \text{ as }k\rightarrow \infty.
\end{equation}
Also, there exists some constant $C_5>0$ independent of $k$ such  that for all $k\ge 1$ and all $x\in E^k$, 
\begin{equation*}
\mathcal{L}_k f(x)\le C_5. 
\end{equation*}
\end{lem}

\begin{proof}
This can be proved by an argument very similar to   \cite[Lemma 2.7]{Lou1}. Thus it is omitted. 
\end{proof}

We prepare the following lemma for the main theorem.

\begin{lem}\label{L:3.13}
Fix $0<\eps<1/64$.  Given any $0<\delta<(1\wedge T)/4$, there exists  $k_\delta \in \mathbb{N}$ such that for all $k\ge k_\delta$, 
\begin{equation*}
\sup_{t\in [2^{-k}/\delta, T]}\IP^{a^*_k}\left[ X^k_t \notin S^k  \right]\le 4C_3\delta,
\end{equation*}
where the $C_3>0$ on the right hand side above is the same as in Corollary \ref{HKUB}.
\end{lem}

\begin{proof}
Given $k_0$ specified in \eqref{def-k0}, let $k_\delta\ge k_0$  be an integer large enough such that $2^{-k_\delta}<\delta^2$. Recall that for any $x\in E^k\backslash S^k$, it must hold that either $x=a^*_k$, or $x\in  D^k_\eps$ and $x\leftrightarrow a^*_k$ in $G^k$. Notice that 
for $t\ge 2^{-k}/\delta$ and $y\leftrightarrow a^*_k$, it holds
\begin{equation}
d(a^*_k, y)\le 1\le 16\cdot 2^kt,
\end{equation}
 it then  follows that for the $C_3$ speicified in Corollary \ref{HKUB}, for any $t\in [2^{-k}/\delta, T]$,
\begin{eqnarray}
\IP^{a^*_{k}}\left[ X_t^k \notin S^k \right] \nonumber
 &\le & \sum_{y\notin S^k} C_3\left( \frac{1}{t} \vee \frac{1}{\sqrt{t}}   \right)  e^{-\frac{d_k(a^*_k, y)^2}{32t}} m_k(dy) \nonumber
\\
\eqref{def-mk}&\le & \left(\sum_{y\in D_\eps^k, y\leftrightarrow a^*_k} \frac{C_3}{t}\cdot 2^{-2k}\right)+\frac{C_3}{t}\cdot m_k(a^*_k) \nonumber
\\
\eqref{P2.3-1}&\le & C_3\delta\cdot 2^k\cdot 2^{-2k}\left( 56\eps \cdot 2^k+28\right)+C_3\cdot \delta\cdot 2^k \cdot 2^{-k}   \nonumber
\\
&\stackrel{\eps<1/64}{\le} & 4C_3\delta.
\end{eqnarray}
The desired conclusion readily follows.
\end{proof}

\begin{thm}
$\{X^k,\IP^{a^*_k},\, k\ge 1\}$ converges weakly to the BMVD described in Theorem \ref{BMVD-non-drift} starting from $a^*$.
\end{thm}

\begin{proof}
This proof is adapted from that for \cite[Proposition 2.9]{Lou1} with some minor changes. We spell out the details for readers' convenience.   Since the laws of $\{X^k\}_{k\ge 1}$ are C-tight in $\mathbf{D}([0, T], E, \rho)$, any sequence has a weakly convergent subsequence supported on the set of continuous paths. Denote by $\{X^{k_j}: j\ge 1\}$ any such  weakly convergent subsequence, and denote by $Y$ its weak limit which is  continuous. By  Skorokhod representation theorem (see, e.g.,  \cite[Chapter 3, Theorem 1.8]{EK}), we may assume that $\{X^{k_j},j\ge 1\}$ as well as $Y$ are defined on a common probability space $(\Omega, \mathcal{F}, \IP)$, so that  $\{X^{k_j},j\ge 1\}$  converges almost surely to $Y$ in the Skorokhod topology.

For every $t\in [0. T]$, we set $\mathcal{M}_t^{k_j} :=\sigma(X^{k_j}_s, s\le t)$  and $\mathcal{M}_t :=\sigma(Y_s, s\le t)$. It is obvious that $\mathcal{M}_t\subset \sigma\{\mathcal{M}_t^{k_j}: j\ge 1\}$. With the class of functions $\mathcal{G}$ defined in \eqref{def-class-G}, in the following we first show that $(Y, \IP^{a^*})$ is a solution to the   $\mathbf{D}([0, T], E, \rho)$  martingale problem $(\mathcal{L}, \mathcal{G})$ with respect to the filtration $\{\mathcal{M}_t\}_{t\ge 0}$. That is,  for  every $f\in \mathcal{G}$, we need to show that 
\begin{equation*}
\left\{f(Y_t)- f(Y_0)-\int_0^t \mathcal{L}f(Y_s)ds\right\}_{t\ge 0}
\end{equation*}
is a martingale with respect to $\{\mathcal{M}_t\}_{t\ge 0}$.
By \cite[Corollary 5.4.1]{FOT}, we know that for any $k\ge 1$  and  any  $f\in \mathcal{G}$,
\begin{equation*}
\left\{f(X^k_t)-f(X_0^k)-\int_0^t\mathcal{\wt{L}}_kf(X^k_s)ds\right\}_{t\ge 0}
\end{equation*}
is a martingale with respect to $\{\mathcal{M}^k_t\}_{t\ge 0}$.  Therefore, for any $0\le t_1<t_2\le T$ and any $A\in \mathcal{M}_{t_1}^{k_j}$, 
it holds for every $j\in \mathbb{N}$ that
\begin{equation}\label{e:3.84}
\IE^{a^*_{k_j}}\left[\left(f(X^{k_j}_{t_2})-f(X^{k_j}_{t_1})-\int_{t_1}^{t_2}\mathcal{L}_{k_j} f(X^{k_j}_s)ds\right)\mathbf{1}_{A}\right]=0.
\end{equation}
We first claim that for any $A\in \mathcal{M}_{t_1}$, 
\begin{equation}\label{compute-T3.13-1}
\lim_{j\rightarrow \infty}\IE^{a^*_{k_j}}\left[\left(f(X^{k_j}_{t_2})-f(X^{k_j}_{t_1})\right)\mathbf{1}_{A}\right] =\IE^{a^*}\left[\left(f(Y_{t_2})-f(Y_{t_1})\right)\mathbf{1}_{A}\right].
\end{equation}
Towards this, we note that it has been claimed at the beginning of this proof  that   one can assume $\{X^{k_j},j\ge 1\}$ as well as $Y$ are defined on a common probability space $(\Omega, \mathcal{F}, \IP)$, so that   $\{X^{k_j},j\ge 1\}$  converges almost surely to $Y$ in the Skorokhod topology, and that $Y$ is a continuous process.   From the proof of \cite[Chapter 3, Theorem 7.8]{EK}), we can tell that: Given a squence  $\{\omega_n, n\ge 1\}$ convergent to $\omega$   in the Skorokhod topology and $\omega$ is continuous at some $t_0>0$, then $\lim_{n\rightarrow \infty}\omega_n(t_0) = \omega (t_0)$. This yields that outside of a zero probability subset of $(\Omega, \FF, \IP)$, $X^{k_j}_t\to Y_t$ as $j\to \infty$ for all $t\in [0, T]$. Thus \eqref{compute-T3.13-1} follows from  dominated convergence theorem.

In order to show the convergence of the integral term in \eqref{e:3.84}, for $k\ge 1$, we denote by
\begin{equation*}
T^k_0=0, \quad \text{and }T^k_l=\inf\{t> T^k_{l-1}:\, X^k_{T_l}\neq X^k_{T_l-}\} \quad \text{for }l=1,2,\dots,
\end{equation*}
i.e., $T^k_l$ is the $l^{\text{th}}$ holding time of $X^k$. $\{T^k_l:l\ge 1\}$ are i.i.d. random variables, each  exponentially distributed with mean $2^{-2k}$. Similar to the proof for \cite[Proposition 2.9]{Lou1}, it holds
\begin{eqnarray}
&&\left|\IE^{a^*_{k_j}}\left[\left(\sum_{l:\; t_1<T^{k_j}_l\le t_2 }\mathcal{L}_{k_j} f\left(X^{k_j}_{T^{k_j}_l}\right)\right)\mathbf{1}_{A}\right] - \IE^m\left[\left(\int_{t_1}^{t_2} \mathcal{L} f(Y_s)ds\right)\mathbf{1}_A\right]\right| \nonumber
\\
&\le &\left| \IE^{a^*_{k_j}}\left[\left(\sum_{l:\; t_1<T^{k_j}_1\le t_2 }\mathcal{L}_{k_j} f\left(X^{k_j}_{T^{k_j}_l}\right)\left(T^{k_j}_l - T^{k_j}_{l-1}\right)-\int_{t_1}^{t_2} \mathcal{L}_{k_j} f (X^{k_j}_s)ds\right)\mathbf{1}_{A}\right] \right|\nonumber
\\
&+& \left|\IE^{a^*_{k_j}}\left[\left(\int_{t_1}^{t_2} \mathcal{L}_{k_j} f (X^{k_j}_s)ds-\int_{t_1}^{t_2} \mathcal{L} f(X^{k_j}_s)ds\right)\mathbf{1}_{A}\right] \right|\nonumber
\\
&+& \left|\IE^{a^*_{k_j}}\left[\left(\int_{t_1}^{t_2} \mathcal{L} f(X^{k_j}_s)ds\right)\mathbf{1}_{A}\right] - \IE^{a^*}\left[\left(\int_{t_1}^{t_2} \mathcal{L} f(Y_s)ds\right)\mathbf{1}_{A}\right] \right| \nonumber
\\
&=&(I)+(II)+(III).\label{compute-T:3.13-2}
\end{eqnarray}
Again by the same reasoning as in \cite[Proposition 2.9]{Lou1}, both (I) and (III) on the right hand side of \eqref{compute-T:3.13-2} converge to zero.  To take care of (II),   for any $\delta>0$, we let $k_\delta$ be chosen as in Lemma \ref{L:3.13}. Hence
\begin{eqnarray}
&&\left|\IE^{a^*_{k_j}}\left[\left(\int_{t_1}^{t_2} \mathcal{L}_{k_j} f (X^{k_j}_s)ds-\int_{t_1}^{t_2} \mathcal{L} f(X^{k_j}_s)ds\right)\mathbf{1}_{A}\right] \right| \nonumber
\\
&\le &\left|\IE^{a^*_{k_j}}\left[\int_{t_1}^{t_2} \mathcal{L}_{k_j} f (X^{k_j}_s)ds-\int_{t_1}^{t_2} \mathcal{L} f(X^{k_j}_s)ds\right] \right|\nonumber
\\
&\le & \Bigg|\IE^{a^*_{k_j}}\bigg[\int_{t_1}^{t_2} \mathcal{L}_{k_j} f (X^{k_j}_s)\mathbf{1}_{\left\{X^{k_j}_s\in S^{k_j}\right\}}ds -\int_{t_1}^{t_2} \mathcal{L} f(X^{k_j}_s)\mathbf{1}_{\left\{X^{k_j}_s\in S^{k_j}\right\}})ds \bigg] \Bigg|\nonumber
\\
&+& \Bigg|\IE^{a^*_{k_j}}\bigg[ \int_{t_1}^{t_2} \mathcal{L}_{k_j} f (X^{k_j}_s)\mathbf{1}_{\left\{X^{k_j}_s\notin S^{k_j}\right\}}ds+\int_{t_1}^{t_2} \mathcal{L} f(X^{k_j}_s)\mathbf{1}_{\left\{X^{k_j}_s\notin S^{k_j}\right\}} ds\bigg] \Bigg|.\label{e:2.40}
\end{eqnarray}
For the first term on the right hand side of \eqref{e:2.40}, by Lemma \ref{L3:12},  we have
\begin{eqnarray}
&& \Bigg|\IE^{a^*_{k_j}}\bigg[\int_{t_1}^{t_2} \mathcal{L}_{k_j} f (X^{k_j}_s)\mathbf{1}_{\left\{X^{k_j}_s\in S^{k_j}\right\}}ds -\int_{t_1}^{t_2} \mathcal{L} f(X^{k_j}_s)\mathbf{1}_{\left\{X^{k_j}_s\in S^{k_j}\right\}})ds \bigg] \Bigg|\nonumber
\\
 &=& \Bigg|\IE^{a^*_{k_j}}\bigg[\int_{t_1}^{t_2} \left(\mathcal{L}_{k_j} f (X^{k_j}_s)- \mathcal{L} f(X^{k_j}_s)\right) \mathbf{1}_{\left\{X^{k_j}_s\in S^{k_j}\right\}}ds\bigg] \stackrel{j\rightarrow \infty}{\rightarrow} 0. \label{e:2.42}
\end{eqnarray}
For the second term on the right hand side of \eqref{e:2.40}, given any $\delta\in (0, (1\wedge T)/4$,  for $k_j\ge k_\delta$ where $k_\delta$ is specified in  Lemma \ref{L:3.13} satisfying $2^{-k_\delta}<\delta^2$,
\begin{eqnarray}
&& \Bigg|\IE^{a^*_{k_j}}\bigg[ \int_{t_1}^{t_2} \mathcal{L}_{k_j} f (X^{k_j}_s)\mathbf{1}_{\left\{X^{k_j}_s\notin S^{k_j}\right\}}ds+\int_{t_1}^{t_2} \mathcal{L} f(X^{k_j}_s)\mathbf{1}_{\left\{X^{k_j}_s\notin S^{k_j}\right\}} ds\bigg] \Bigg| \nonumber
\\
&\le &\Bigg|\IE^{a^*_{k_j}}\bigg[ \int_{0}^{T} \mathcal{L}_{k_j} f (X^{k_j}_s)\mathbf{1}_{\left\{X^{k_j}_s\notin S^{k_j}\right\}}ds+\int_{0}^{T} \mathcal{L} f(X^{k_j}_s)\mathbf{1}_{\left\{X^{k_j}_s\notin S^{k_j}\right\}} ds\bigg] \Bigg| \nonumber
\\
&= & \Bigg|\IE^{a^*_{k_j}}\bigg[ \int_{0}^{2^{-k}/\delta} \mathcal{L}_{k_j} f (X^{k_j}_s)\mathbf{1}_{\left\{X^{k_j}_s\notin S^{k_j}\right\}}ds+\int_{0}^{2^{-k}/\delta} \mathcal{L} f(X^{k_j}_s)\mathbf{1}_{\left\{X^{k_j}_s\notin S^{k_j}\right\}} ds\bigg] \Bigg| \nonumber
\\
&+& \Bigg|\IE^{a^*_{k_j}}\bigg[ \int_{2^{-k}/\delta} ^T\mathcal{L}_{k_j} f (X^{k_j}_s)\mathbf{1}_{\left\{X^{k_j}_s\notin S^{k_j}\right\}}ds+\int_{2^{-k}/\delta}^T \mathcal{L} f(X^{k_j}_s)\mathbf{1}_{\left\{X^{k_j}_s\notin S^{k_j}\right\}} ds\bigg] \Bigg| \nonumber
\\
(\text{Lemma }\ref{L3:12})&\le &  \left(C_5+\|\mathcal{L}f\|_\infty\right) \cdot \frac{2^{-k}}{\delta} + T\cdot \left(C_5+\|\mathcal{L}f\|_\infty\right) \sup_{t\in [2^{-k}/\delta, T]}\IP^{a^*_{k_j}}\left[   X_t^{k_j}\notin S^{k_j}\right] \nonumber
\\
(\text{Lemma  } \ref{L:3.13})  &\le &  \left(C_5+\|\mathcal{L}f\|_\infty\right) \delta + T\cdot \left(C_5+\|\mathcal{L}f\|_\infty\right) \cdot 4C_1\delta.
\end{eqnarray}
Since $\delta$ can be made arbitrarily small, we have proved that for the second term on the right hand side of \eqref{e:2.40}, it also holds   that 
\begin{equation}\label{e:2.43}
\lim_{j\rightarrow \infty } \Bigg|\IE^{a^*_{k_j}}\bigg[ \int_{t_1}^{t_2} \mathcal{L}_{k_j} f (X^{k_j}_s)\mathbf{1}_{\left\{X^{k_j}_s\notin S^{k_j}\right\}}ds+\int_{t_1}^{t_2} \mathcal{L} f(X^{k_j}_s)\mathbf{1}_{\left\{X^{k_j}_s\notin S^{k_j}\right\}} ds\bigg] \Bigg|=0.
\end{equation}
Combining \eqref{e:2.42} and \eqref{e:2.43}, we have showed that (II) on the right hand side of \eqref{compute-T:3.13-2}, thus the entire right hand side of \eqref{compute-T:3.13-2} tends to zero as $j\rightarrow \infty$. This combined with \eqref{compute-T3.13-1} shows that $(Y, \IP^{a^*})$ is indeed a solution to the   $\mathbf{D}([0, T], E, \rho)$  martingale problem $(\mathcal{L}, \mathcal{G})$   with respect to the filtration $\{\mathcal{M}_t\}_{t\ge 0}$.

To finish the proof, it remains to show that  the $\mathbf{D}([0, T], E, \rho)$  martingale problem $(\mathcal{L}, \mathcal{G})$  has a unique solution. Towards this, we denote the infinitesimal generator of BMVD defined in Theorem \ref{BMVD-non-drift} by $(\mathcal{L}, \mathcal{D(L)})$ (see \cite[Theorem 2.3]{CL} for its description). In view of the definition of $\mathcal{G}$, it is easy to verify that the bp-closure (whose definition can be found, e.g., in \cite[Definition 3.4.3]{AB}) of the graph  of $\mathcal{L}$ restricted on $\mathcal{D(L)}\cap \mathcal{G}$ is the same as the bp-closure of  the graph of $\mathcal{L}$ restricted on  $\mathcal{D(L)}\cap C_c(E)$.  Therefore by \cite[Proposition 3.4.19]{AB},  the $\mathbf{D}([0, T], E, \rho)$  martingale problem $(\mathcal{L}, \mathcal{D(L)}\cap \mathcal{G})$ has the same set of solution(s) as the  $\mathbf{D}([0, T], E, \rho)$  martingale problem $(\mathcal{L}, \mathcal{D(L)}\cap C_c(E))$.

Finally, given any $f \in C_0(E)\cap \mathcal{D(L)}$, there exists $\{f_n\}_{n\ge 1}\subset C_c(E)\cap \mathcal{D(L)}$ such that $f_n\rightarrow f$ and $\mathcal{L}f_n\rightarrow \mathcal{L}f$ both in $L^2$-norm. This means that the closure of $\mathcal{L}$   restricted   on $\mathcal{D(L)}\cap C_c(E)$ coincides with the closure of $\mathcal{L}$ restricted on $\mathcal{D(L)}\cap C_0(E)$, which corresponds to the BMVD defined in Theorem \ref{BMVD-non-drift} which is a Feller process with strong Feller property. By \cite[Theorem 3.1, Remark 3.3]{MF}, the $\mathbf{D}([0, T], E, \rho)$  martingale problem $(\mathcal{L}, \mathcal{D(L)}\cap C_c(E))$  has a unique solution, which has to be the BMVD defined in Theorem \ref{BMVD-non-drift}. In view of the last sentence of the last paragraph,  BMVD defined in Theorem \ref{BMVD-non-drift} also has to be the $\mathbf{D}([0, T], E, \rho)$  martingale problem $(\mathcal{L}, \mathcal{D(L)}\cap \mathcal{G})$. Since earlier in thie proof we have claimed that $(Y, \IP^{a^*})$ is a solution to the   $\mathbf{D}([0, T], E, \rho)$  martingale problem $(\mathcal{L}, \mathcal{G})$   with respect to the filtration $\{\mathcal{M}_t\}_{t\ge 0}$,  $(Y, \IP^{a^*})$ coincides with the BMVD defined in Theorem \ref{BMVD-non-drift} starting from $a^*$. Since $Y$ is the sequencial limit of any weakly convergent subsequence of $\{X^k\}_{k\ge 1}$, the proof is complete. 
\end{proof}

\vskip 0.3truein

\noindent {\bf Shuwen Lou}

\smallskip \noindent
Department of Mathematics and Statistics, Loyola University Chicago,
\noindent
Chicago, IL 60660, USA

\noindent
E-mail:  \texttt{slou1@luc.edu}

 \end{document}